\UseRawInputEncoding
\pdfoutput=1
\documentclass[preprint,12pt,sort&compress,numbers]{elsarticle}
\usepackage{amsfonts}
\usepackage{amsmath}
\usepackage{amssymb}
\usepackage{amsthm}
\usepackage{mathrsfs}
\usepackage{setspace}
\usepackage{bookmark}

\input amssym.def
\input amssym.tex

\date{}

\theoremstyle{plain}
\numberwithin{equation}{section}
\newtheorem{Theorem}{Theorem}[section]

\newtheorem{Lemma}{Lemma}[section]

\theoremstyle{definition}
\newtheorem{Remark}{Remark}[section]

\newcommand{\definition}{{\lower .5ex
\hbox{$\>\>\stackrel{\triangle}{=}\>\>$} }}
\newcommand\diverg{\mathop{\mbox{\rm div}}}
\newcommand\rot{\mathop{\mbox{\rm rot}}}

\allowdisplaybreaks[4]

\begin{document}

\baselineskip=24pt

\leftskip 0 true cm
\rightskip 0 true cm

\newpage
\begin{center}
{\large \bf Global unique solution for the 3-D full compressible MHD equations in space of lower regularity $^*$}\
\
\footnote{$^*$This work was supported by National Natural Science Foundation of China  [grant number 12101345],  Natural Science Foundation of Shandong Province of China [grant number ZR2021QA017].}
\footnote{$^{**}$
Corresponding Author: feichenstudy@163.com (F. Chen)}
\footnote{$^{\dag}$
Email Address:
 wcb1216@163.com (C.B. Wang);\
 feichenstudy@163.com (F. Chen);\
 shuai172021@163.com (S. Wang). }

 Chuanbao Wang, Fei Chen$^{**}$ \ and\  Shuai Wang$^\dag$ \\[2ex]
 School of Mathematics and Statistics,
 Qingdao University,
\\[0.5ex]
 Qingdao, Shandong 266071, P. R. China\\[2ex]
\end{center}

\bigskip

\centerline{\bf Abstract}
In this paper, we establish new $L^p$ gradient estimates of the solutions in order to discuss Cauchy problem for the full compressible magnetohydrodynamic(MHD) systems in $\mathrm{R}^3$.
We use the ``$\diverg-\rm{curl}$" decomposition technique  (see \cite{{HJR},{MR}}) and   new modified  effective viscous flux and  vorticity to calculate ``$\Vert\nabla \mathbf{u}\Vert_{L^3}$" and ``$\Vert\nabla \mathbb{H}\Vert_{L^3}$".
As a result, we obtain  global well-posedness for the solution  with the initial data being in a class of space with lower regularity, while the energy of which  should be suitably small.

\bigskip
\leftskip 0 true cm
\rightskip 0 true cm

\section{\large\bf Introduction}

\setcounter{Theorem}{0}
\setcounter{Lemma}{0}

\setcounter{section}{1}

\label{1} \setcounter{Theorem}{0} \setcounter{Lemma}{0}
\setcounter{section}{1}
In this paper, we focus on the following 3-D full compressible magnetohydrodynamic(MHD) equations:
\begin{eqnarray}\label{MHD}
\begin{cases}
\varrho_s +\nabla\cdot(\varrho \mathbf{u})=0,\\
\varrho (\mathbf{u}_s+\mathbf{u}\cdot\nabla \mathbf{u})+\nabla \mathfrak{P}=\iota \Delta \mathbf{u}
+(\iota+\aleph)\nabla(\nabla\cdot \mathbf{u}) +(\rot \mathbb{H}) \times \mathbb{H},\\
\mathbb{C}_v\varrho(\Phi_s +\mathbf{u}\cdot\nabla\Phi)+\mathfrak{P}(\nabla\cdot \mathbf{u})=\kappa \Delta\Phi
+2\iota\vert\mathfrak{D}(\mathbf{u})\vert^2+\aleph  (\nabla\cdot \mathbf{u})^2+\nu\vert\rot \mathbb{H}\vert^2,\\
\mathbb{H}_s-\rot(u \times \mathbb{H})=-\rot(\nu\rot \mathbb{H}), \quad \nabla\cdot\mathbb{H}=0,\\
(\varrho,\mathbf{u},\Phi,\mathbb{H})\rightarrow(1,0,1,0)\quad as\quad \vert x\vert\rightarrow\infty, s>0,\quad(\varrho,\mathbf{u},\Phi,\mathbb{H})\vert_{s=0}=(\varrho_0,\mathbf{u}_0,\Phi_0,\mathbb{H}_0),
\end{cases}
\end{eqnarray}
where $( x,s ) \in \mathrm{R}^3 \times ( 0,+\infty )$.
The  magnetic field,  temperature, density and  velocity of the fluid    are expressed as   $\mathbb{H}=(\mathbb{H}_{1},\mathbb{H}_{2},\mathbb{H}_{3})$, $\Phi$, $\varrho$, $\mathbf{u}=(\mathbf{u}_{1},\mathbf{u}_{2}, \mathbf{u}_{3})$, respectively. $\iota>0,\ 2\iota+3\aleph\ge0,$ ( $\iota$ and $\aleph$ are viscosity coefficients). $\kappa$ and $\nu$ represent heat conductivity  and magnetic diffusion coefficients.
In addition, the adiabatic exponent $\gamma>1$, perfect gas constant $\mathcal{R}>0$, internal energy $e$ and specific heat capacity constant $\mathbb{C}_v>0$,
$$\mathfrak{P}(\varrho,e)=\mathcal{R}\varrho\Phi=(\gamma-1)\varrho e, \quad e=\frac{\mathcal{R}\Phi}{\gamma-1}=\mathbb{C}_v\Phi,$$
and the viscous stress tensor $\mathfrak{D}(\mathbf{u})$ is expressed as
\begin{align}\label{D(u)}
\mathfrak{D}(\mathbf{u})=\frac{(\nabla \mathbf{u})+(\nabla \mathbf{u})'}{2}.
\end{align}

The study of Cauchy problem for the compressible MHD systems is dynamic.  The compressible MHD systems is described by equations of density, velocity, magnetic  and temperature. In the first place, Kawashima \cite{SS} explored the global existence and uniqueness of classical solutions to the Cauchy problem of  multi-dimensional full compressible MHD systems with the initial data being suitably small. Fan and Yu \cite{JWS} obtained the existence of unique local strong solutions  in a bounded domain  $\Omega\subset\mathrm{R}^3$ in case of $\rho_0\geq 0$. Li, Xu and Zhang \cite{HXJG} obtained the global well-posedness about fluid density including vacuum in 3-D isentropic compressible case with small initial energy. On the basis of their study, Liu and Zhong \cite{YXG} expanded to discuss strong solutions of MHD equations with vacuum, when density and  magnetic field satisfy ``$\Vert\varrho_0\Vert_{L^\infty}+\Vert \mathbb{H}_0\Vert_{L^3}$" is suitably small with the large initial energy. Recently, Hou, Jiang and Peng \cite{XMHG} studied the strong solutions under the situations of large oscillations and vacuum with small data``$\Vert\varrho_0\Vert_{L^1}+\Vert \mathbb{H}_0\Vert_{L^2}$". The full compressible Hall magnetohydrodynamic systems was studied by Tao and Zhu in \cite{QCG},  in which they studied the large time behavior of the global solution. Feireisl and Li \cite{EYO} explored non-uniqueness of weak solutions with large data to magnetic field, and they also gained the similar consequence for viscid heat-conductive fluid. Moreover, Fan and Li \cite{JFG} obtain the global strong solution by the constructions of regularity criterion and bootstrap argument for non-isentropic cases without magnetic diffusion in bounded domain.
For the incompressible MHD equations, Bie, Wang and Yao \cite{QQZG} obtained a global and unique solution with variable density in case of the initial data ``$(\varrho_0-1,\mathbf{u}_0,\mathbb{H}_0)$" satisfying certain smallness assumption in Besov Space. Lately, Zhao \cite{XG} established the global well-posedness results with a ion-slip effects small data, and they also get the decay rate for higher order derivatives of solutions. Then, Zhai, Li and Zhao \cite{XZG-2} obtained the global smooth solution of  inviscid Hall MHD equations with small initial data in 3-D period region.

When ignoring the effection of magnetic,  the MHD equations  turns into the Navier-Stokes equations. Hoff \cite{DD} considered the synchronizing global existence to the 2-D and 3-D compressible situations in early time with potentially discontinuous initial data. Huang et al. \cite{{XJG},{XJYS},{XJZS},{XJZG}} engaged a range of Cauchy problems with small energy including vacuum and non-vacuum situations, which included  large oscillations of the energy and initial conditions. With respect to the large time characteristics and decay rate of the solutions with vanishing of density into infinity, Wen and Zhu \cite{HCG} studied the global strong and classical solutions with changing elements of the decay rate in the full compressible situation. Moreover, the initial vacuum  were taken into account in \cite{{HCG},{XJZS},{XJZG},{XJG}}. For inhomogeneous incompressible cases, Abidi, Gui and Zhang \cite{HGPW} obtained the global well-posedness in critical Besov space. Afterward, Zhai and Yin \cite{XZG} constructed the comparative faintish regularity index about the initial velocity, which also expanded to the inhomogeneous incompressible case simultaneously.

Recently, Xu and Zhang \cite{HJR} established the global uniqueness of solution for the full compressible Navier-Stokes systems in a space with lower regularity compared with that of \cite{AP} in which the initial data need to be in $ H^2$.
 It's still an interesting and challenging thing to explore whether we can establish the global uniqueness for solution of full compressible MHD equations(\ref{MHD}) with lower regularity initial conditions compared to $H^2(R^3)$.

Before presenting our conclusion, we  emphasize some important  notations. Define the material derivative, modified effective viscous flux and the vorticity as $\mathscr{L}$, $\mathbb{G}$ and $\mathscr{N}$,
\begin{eqnarray}\label{modified derivatives}
\begin{cases}
\dot{\mathscr{L}} \stackrel{\triangle}{=} \mathscr{L}_s+\mathbf{u}\cdot\nabla \mathscr{L},
\\\mathbb{G} \stackrel{\triangle}{=}(2\iota+\aleph)(\nabla\cdot \mathbf{u})-\mathcal{R}(\varrho\Phi-1)
-(-\Delta)^{-1} \nabla\cdot \Big[(\rot \mathbb{H})\times \mathbb{H} \Big],
\\\mathscr{N} \stackrel{\triangle}{=}\rot \mathbf{u}
-(-\Delta)^{-1}\iota^{-1}\Big[\rot\big((\rot \mathbb{H})\times \mathbb{H}\big)\Big],
\end{cases}
\end{eqnarray}
which are similar to \cite{{DGJ},{XJG}}. But $\mathbb{G}$ and $\mathscr{N}$ are influenced by magnetic field. Besides, one may check that from (\ref{MHD}) and (\ref{modified derivatives}):
\begin{eqnarray}\label{Material derivative transformation}
\Delta\mathbb{G}=\nabla\cdot{(\varrho\dot{\mathbf{u}})},\quad\iota\Delta\mathscr{N}
=\rot(\varrho\dot{\mathbf{u}}),
\end{eqnarray}
and we denote the initial energy as below:
\begin{align}
I_0&\stackrel{\triangle}{=}\int\Big(\frac{1}{2}\big(\varrho_0\vert \mathbf{u}_0\vert^2+\vert \mathbb{H}_0\vert^2\big)
+\mathcal{R}\big(\varrho_0\log\varrho_0+1-\varrho_0\big)+\mathbb{C}_v\varrho_0\big(\Phi_0-\log\Phi_0-1 \big)\Big)dx.\label{initial energy}
\end{align}

 The main results of this paper are shown below:
\begin{Theorem}\label{Main result}
Let $m\in[\frac{9}{2},6)$, suppose that
\begin{eqnarray}\label{(1.17**)}
\begin{cases}
\inf\varrho_0(x)>0,\quad\varrho_0-1\in H^1\cap W^{1,m},\\
\inf\Phi_0(x)>0,\quad\Phi_0-1\in L^{\infty}\cap H^1,\\
(\mathbf{u}_0,\mathbb{H}_0)\in H^1\cap W^{1,3},\\
\end{cases}
\end{eqnarray}
$\exists$ a constant $\zeta>0$ relying only on $\iota$, $\aleph$, $\kappa$, $\mathcal{R}$, $\gamma$, $\nu$, $\inf\varrho_0$, $\sup\varrho_0$, $\inf\Phi_0$, $\sup\Phi_0$, $\Vert\nabla \mathbf{u}_0\Vert_{L^2}$ and $\Vert\nabla \mathbb{H}_0\Vert_{L^2}$, for any $S\in(0,\infty)$, such that if
\begin{eqnarray}\label{(1.18**)}
I_0\le\zeta,
\end{eqnarray}
the problem $(\ref{MHD})$ has a unique global solution $(\varrho,\mathbf{u},\Phi, \mathbb{H})$ satisfying
\begin{eqnarray}\label{(1.19**)}
\begin{cases}
\varrho-1\in C([0,S];H^1\cap W^{1,m}),\quad\inf\varrho_(x,s)>0,\\
\Phi-1\in L^{\infty}(0,S;H^1)\cap L^2(0,S;H^2)\cap L^n(0,S;L^{\infty}),\quad\inf\Phi(x,s)>0,\\
(\mathbf{u},\Phi-1,\mathbb{H})\in C([0,S];L^2\cap L^l),\\
(\mathbf{u},\mathbb{H})\in L^{\infty}(0,S;H^1\cap W^{1,3})\cap L^2(0,S;H^2)\cap L^n(0,S;W^{1,\infty}),\\
(\sqrt{s}\dot{\Phi},\sqrt{s}\dot{\mathbf{u}},\sqrt{s}\dot{\mathbb{H}})\in L^{\infty}(0,S;L^2),\\
(\sqrt{s}\nabla\dot{\Phi},\sqrt{s}\nabla\dot{\mathbf{u}},\sqrt{s}\nabla\dot{\mathbb{H}})\in L^2(0,S;L^2),
\end{cases}
\end{eqnarray}
where $l\in[2,6)$ and $1\le n\le\frac{4m}{5m-6}$.
\end{Theorem}
\begin{Remark}
We emphasis that the spaces which the initial data belong to are weaker that $H^2$.
\end{Remark}

\begin{Remark}
 The result in Theorem \ref{Main result}  covers the conclusion of  that for isentropic compressible MHD systems by Zhang \cite{MR} that  bases on \cite{HXJG} and \cite{HJR}, once the pressure is  treated properly.
\end{Remark}

Fortunately, based on the previous results (see \cite{{XJG},{XJZG},{CFSA}}), we establish the lower order time-independent global a priori estimates to (\ref{MHD}), and we needn't introduce the analogous details of the process below (see Lemma 3.1).

And it's also difficult to achieve the estimates of $``\Vert\nabla\Phi\Vert_{L^2}"$, $``\Vert\nabla\mathbf{u}\Vert_{L^3}"$ and $``\Vert\nabla\mathbb{H}\Vert_{L^3}"$.
On the basis of the previous research (see \cite{{HJR},{MR}} ), we can deal with $\nabla(\varrho^{-1}\times(\ref{MHD})_2)$  from the outcomes of Lemma \ref{L (2.1**)}-\ref{(2.5**)}.
The key of our program is to cope with ``$\Vert(\nabla\cdot\mathbf{u})\Vert_{L^m}$", ``$\Vert(\nabla\cdot\mathbb{H})\Vert_{L^m}$", ``$\Vert\rot\mathbf{u}\Vert_{L^m}$" and ``$\Vert\rot\mathbb{H}\Vert_{L^m}$". Operating ``$\nabla\cdot$" to $(\ref{MHD})_2$ and $(\ref{MHD})_4$, we gain
\begin{align}\label{(1.21**)u}
&\varrho(\nabla\cdot \mathbf{u})_s+\varrho\mathbf{u}\cdot\nabla(\nabla\cdot \mathbf{u})-(2\iota+\aleph)\Delta(\nabla\cdot\mathbf{u})\notag\\
&\quad=-(\nabla\varrho)\cdot\mathbf{u}_s-\partial_j(\varrho\mathbf{u}^i)\partial_i\mathbf{u}^j
-\Delta\mathfrak{P}(\varrho,e)+\partial_j\mathbb{H}^i\partial_i\mathbb{H}^j-\frac{1}{2}\Delta\vert\mathbb{H}\vert^2,
\end{align}
and
\begin{align}\label{(1.21**)H}
(\nabla\cdot\mathbb{H})_s-\nu\Delta(\nabla\cdot\mathbb{H})
=\nabla\cdot(\mathbb{H}\cdot\nabla\mathbf{u}-\mathbf{u}\cdot\nabla\mathbb{H}-\mathbb{H}(\nabla\cdot\mathbf{u})).
\end{align}
Practically, multiplying (\ref{(1.21**)u}) by $\vert(\nabla\cdot \mathbf{u})\vert(\nabla\cdot \mathbf{u})$ and (\ref{(1.21**)H}) by $\vert(\nabla\cdot\mathbb{H})\vert(\nabla\cdot\mathbb{H})$, after adding up the resulting equations and integrating by parts over $\mathrm{R}^3$, we gain the estimates of ``$\Vert(\nabla\cdot\mathbf{u})\Vert_{L^3}$" and ``$\Vert(\nabla\cdot\mathbb{H})\Vert_{L^3}$". Likewise, ``$\Vert\rot\mathbf{u}\Vert_{L^3}$" and ``$\Vert\rot\mathbb{H}\Vert_{L^3}$" can also be achieved.

The major intention of our paper is restricted to the lower regularity $(\ref{(1.17**)})_3$ (i.e. $\mathbb{H}_0\in H^1\cap W^{1,3}$ but $\mathbb{H}_0\notin H^2$), so that we can't gain these important estimates (i.e. $\mathbb{H}_s\in L^{\infty}(0,S;L^2)\cap L^2(0,S;H^1) $) instantly. So we should figure out $\Vert\dot{\mathbb{H}}\Vert_{L^n (0,S;L^m)}$ and other estimates with $m$ and $n$ as in (\ref{(2.33**)}).
At the same time, from $(\ref{MHD})_2$,
\begin{align}
\mathbf{u}_s=\varrho^{-1}\Big(\iota \Delta \mathbf{u}
+(\iota+\aleph)\nabla(\nabla\cdot \mathbf{u}) +(\rot \mathbb{H}) \times\mathbb{H}-\nabla\mathfrak{P}(\varrho,\Phi)\Big)-\mathbf{u}\cdot\nabla{\mathbf{u}},
\end{align}
and for smooth scalar/vetor functions $\mathrm{O}$, $\mathsf{P}$ and $\mathrm{Q}$,
\begin{align}\label{(1.27**)}
\int(\nabla\mathrm{O}\times\nabla\mathrm{P})\cdot\mathrm{Q} dx
=-\int\mathrm{P}(\nabla\mathrm{O})\cdot(\rot\mathrm{Q}) dx,
\end{align}
we can close the estimates of $\Vert\nabla\mathbf{u}\Vert_{L^3}$, $\Vert\nabla\mathbb{H}\Vert_{L^3}$, and $\Vert\nabla\Phi\Vert_{L^2}$ (see [Lemma \ref{Lbbb}], \cite[Lemma 2.7]{HJR} and \cite[Lemma 2.7]{MR}).
In terms of these consequences, we can derive the estimates of temperature (see [Lemma \ref{(2.8**)}]), which are vital to the consideration of uniqueness.

The rough structure of our research is expressed here: Firstly, we will briefly introduce a variety of inequalities and preliminary knowledge. In Section 3, the global a priori estimates can be established. Eventually, we give a brief proof to Theorem \ref{Main result}.

\section{\large\bf Preliminaries}

\begin{Lemma}(Gagliardo-Nirenberg inequality)\label{Gagliardo-Nirenberg Inequality}(cf. \cite{LO})
Suppose that $1\le k,s,t\le\infty$,$0\le t<m$ and $\frac{r}{p}\le\alpha\le1$, satisfy
\begin{align*}
\frac{1}{k}-\frac{r}{q}=\alpha\Big(\frac{1}{s}-\frac{p}{q} \Big)+(1-\alpha)\frac{1}{t},
\end{align*}
where $p-r-\frac{q}{s}\in 0\cup Z^{+}$ and $\frac{r}{p}\le \alpha<1$. $\exists$ a constant $I=I(p,q,k,r,s,t)$, such that
\begin{align}\label{GN11}
\Vert\nabla^r\varphi \Vert_{L^k}
\le I\Vert\nabla^p\varphi\Vert_{L^s}^\alpha\Vert\varphi\Vert_{L^t}^{(1-\alpha)},
\end{align}
for all $\varphi\in W^{p,s}(\mathrm{R}^q)\cap L^t(\mathrm{R}^q)$. Particularly, for $q=3$, we have
\begin{eqnarray}\label{Special GN11}
\Vert\nabla\varphi \Vert_{L^6}\le I\Vert\nabla^2\varphi \Vert_{L^2},
\end{eqnarray}

\end{Lemma}

The following consequences referred from \cite[Lemma 2.3]{XJZG} will be essential to our prove.
\begin{Lemma}\label{four inequality}
Let $(\varrho,\mathbf{u},\Phi,\mathbb{H})$ be a smooth solution of $(\ref{MHD})$, $\exists$ a positive constant $I>0$ relying on $\iota$, $\aleph$ and $\mathcal{R}$ for $m\in[2,6]$, the estimates can be obtained:
\begin{align}
&\Big(\Vert\nabla \mathbb{G}\Vert_{L^m}+\Vert\nabla \mathscr{N}\Vert_{L^m}\Big)
\le I\Vert\varrho\dot{\mathbf{u}} \Vert_{L^m},\label{g,n estimates}\\
&\Big(\Vert \mathbb{G}\Vert_{L^m}+\Vert \mathscr{N}\Vert_{L^m}\Big)\le I\Vert \varrho \dot{\mathbf{u}}\Vert_{L^2}^{\frac{3m-6}{2m}}\Big(\Vert \varrho\Phi-1\Vert_{L^2}
+\Vert \nabla \mathbf{u}\Vert_{L^2}+\Vert \vert\mathbb{H}\vert^2\Vert_{L^2}\Big)^{\frac{6-m}{2m}},\label{(2.27*)}\\
&\Vert \nabla \mathbf{u}\Vert_{L^m}\le I\Big(\Vert \mathbb{G}\Vert_{L^m}
+\Vert \mathscr{N}\Vert_{L^m}+\Vert \varrho\Phi-1\Vert_{L^m}
+\Vert \vert \mathbb{H}\vert^2\Vert_{L^m}\Big),\label{(2.25*)}\\
&\Vert \nabla \mathbf{u}\Vert_{L^m}\le I\Vert \nabla \mathbf{u}\Vert_{L^2}^{\frac{6-m}{2m}}
\Big(\Vert \varrho\dot{\mathbf{u}}\Vert_{L^2}+\Vert \varrho\Phi-1\Vert_{L^6}+\Vert \mathbb{H}\cdot\nabla\mathbb{H}\Vert_{L^2}^2\Big)^{\frac{3m-6}{2m}}.\label{(2.28*)}
\end{align}
\end{Lemma}

\begin{Lemma}(Beale-Kato-Majda type inequality)(cf.\cite{{JTAR},{XJZS}})
Suppose that $h\in(3,+\infty)$ and $\nabla\mathrm{F}\in L^2(\mathrm{R}^3)\cap D^{1,h}(\mathrm{R}^3)$, $\exists$ a constant $I=I(h)>0$, the following estimate can be obtained:
\begin{align}
\Vert\nabla\mathrm{F}\Vert_{L^{\infty}(\mathrm{R}^3)}
&\le I\Big(\Vert\nabla\cdot\mathrm{F}\Vert_{L^{\infty}(\mathrm{R}^3)}
+\Vert\rot\mathrm{F}\Vert_{L^{\infty}(\mathrm{R}^3)}\Big)
\log\Big(e+\Vert\nabla^2\mathrm{F}\Vert_{L^h(\mathrm{R}^3)}\Big)\notag\\
&\quad+I\Big(1+\Vert\nabla\mathrm{F}\Vert_{L^2(\mathrm{R}^3)}\Big).\label{B-K-M inequality}
\end{align}
\end{Lemma}

\section{\large\bf A priori estimates }
The segment will introduce the derivations of a priori estimates for (\ref{MHD}) on $\mathrm{R}^3\times[0,S]$.
The following outcomes derived form \cite{CFSA} play an important role to our proof.
\begin{Lemma}\label{L (2.1**)}
Let $\bar{\varrho}>2$ , $\bar{\Phi}>1$, and $\Omega\ge0$ (not necessarily small). Assume that
\begin{eqnarray}\label{(2.2**)}
\begin{cases}
0<\inf{\varrho_0}\le \sup{\varrho_0}<\bar{\varrho},\\
0<\inf{\Phi_0}\le \sup{\Phi_0}<\bar{\Phi},\\
\Vert\nabla \mathbf{u}_0\Vert_{L^2}, \Vert\nabla \mathbb{H}_0\Vert_{L^2}\le\Omega.\\
\end{cases}
\end{eqnarray}
$\exists$ constants $\mathcal{Q},\zeta>0$, relying on $\iota$, $\aleph$, $\kappa$, $\mathcal{R}$, $\gamma$, $\nu$, $\bar{\varrho}$, $\bar{\Phi}$ and $\Omega$, such that
if
\begin{align}\label{(2.3**)}
I_0\le \zeta,
\end{align}
where $I_0$ has been defined in (\ref{initial energy}), the following estimates can be obtained:
\begin{align}
&0<\varrho(x,s)\le2\bar{\varrho},\quad \forall(x,s)\in\mathrm{R}^3\times[0.S]\label{(2.4**)}\\
&\sup\limits_{s\in [0,S]}\Big(\Vert \nabla \mathbf{u}\Vert_{L^2}^2 +\Vert \nabla \mathbb{H}\Vert_{L^2}^2 \Big)
+\int_0^S\int\Bigl(\varrho\vert \dot{\mathbf{u}} \vert^2+\vert \dot{\mathbb{H}} \vert^2\Big)dxds \le\mathcal{Q},\label{(2.5**)}\\
&\sup\limits_{s \in [0,S]}
\Big(\Vert\varrho-1\Vert_{L^2}^2+\Vert\sqrt{\varrho}\mathbf{u}\Vert_{L^2}^2+\Vert\mathbb{H}\Vert_{L^2}^2
+\Vert\sqrt{\varrho}(\Phi-1)\Vert_{L^2}^2\Big)\notag\\
&\quad+\int_0^S \Big(\Vert\nabla\Phi\Vert_{L^2}^2+\Vert \nabla \mathbf{u}\Vert_{L^2}^2
+\Vert \nabla \mathbb{H}\Vert_{L^2}^2 \Big)ds\le\mathcal{Q}I_0^{\frac{1}{4}},\label{aaaa}\\
&\sup\limits_{s \in [0,S]}\Big(\psi^2\Vert \nabla \Phi\Vert_{L^2}^2+\psi(\Vert\nabla \mathbf{u}\Vert_{L^2}^2
+\Vert\nabla \mathbb{H}\Vert_{L^2}^2)+\psi^2\int(\varrho\vert\dot{\mathbf{u}}\vert^2
+\vert\dot{\mathbb{H}}\vert^2)dx \Big)\notag\\
&\quad+\int_0^S\int\Big(\psi(\varrho \vert \dot{\mathbf{u}} \vert^2+\vert \dot{\mathbb{H}}\vert^2)
+\psi^2(\vert \nabla\dot{\mathbf{u}} \vert^2+\vert \nabla\dot{\mathbb{H}} \vert^2) +\psi^2\varrho\vert\dot{\Phi}\vert^2 \Big)dxds\le\mathcal{Q}I_0^{\frac{1}{6}},\label{bbb}\\
&\sup\limits_{s \in [0,S]}\psi^4 \int \varrho\vert \dot{\Phi}\vert^2dx
+\int_0^S\int\psi^4\vert \nabla \dot{\Phi} \vert^2dxds \le\mathcal{Q}I_0^{\frac{1}{8}}.\label{(2.8**)}
\end{align}
where $\psi(s)\stackrel{\triangle}{=} \min \{ 1,s \} $.
\end{Lemma}

\begin{Lemma}\label{L(2.2**)}
On the premise of Lemma $3.1$, $\exists$ a constant $I>0$ relying on $S$ such that the following estimate can be obtained:
\begin{align}
&\sup_{0\le s\le S}s\Big(\Vert\nabla\Phi\Vert_{L^2}^2+\Vert\varrho^{\frac{1}{2}}\dot{\mathbf{u}}\Vert_{L^2}^2
+\Vert\dot{\mathbb{H}}\Vert_{L^2}^2\Big)\notag\\
&\quad+\int_0^Ss\Big(\Vert\varrho^{\frac{1}{2}}\dot{\Phi}\Vert_{L^2}^2+\Vert\nabla^2\Phi\Vert_{L^2}^2
+\Vert\nabla\dot{\mathbf{u}}\Vert_{L^2}^2+\Vert\nabla\dot{\mathbb{H}}\Vert_{L^2}^2\Big)ds\le I(S).\label{(2.10**)}
\end{align}
\end{Lemma}

\begin{proof}
First of all, operating $s\dot{\mathbf{u}}^j[\frac{\partial}{\partial s}+\nabla\cdot (\mathbf{u}\cdot)]$ to $(\ref{MHD})_2^j$,
$s\dot{\mathbb{H}}^j[\frac{\partial}{\partial s}+\nabla\cdot (\mathbf{u}\cdot)]$ to $(\ref{MHD})_4^j$, adding up and integrating the resulting equations over $\mathrm{R}^3$, one has
\begin{align}
&\frac{1}{2}\frac{d}{ds}\Big(\int s(\varrho\vert\dot{\mathbf{u}}\vert^2+\vert\dot{\mathbb{H}}\vert^2)dx\Big)\notag\\
&\quad=\frac{1}{2}\int (\varrho\vert\dot{\mathbf{u}}\vert^2+\vert\dot{\mathbb{H}}\vert^2)dx
-\int s\dot{\mathbf{u}}^j\Big(\partial_j\mathfrak{P}_s+\nabla\cdot(\mathbf{u}\partial_j\mathfrak{P})\Big)dx\notag\\
&\qquad+\iota\int s\dot{\mathbf{u}}^j\Big(\Delta \mathbf{u}_s^j+\nabla\cdot(\mathbf{u}\Delta \mathbf{u}^j)\Big)dx\notag\\
&\qquad+(\iota+\aleph)\int s\dot{\mathbf{u}}^j\Big(\partial_s\partial_j(\nabla\cdot \mathbf{u})
+\nabla\cdot(\mathbf{u}\partial_j(\nabla\cdot \mathbf{u}))\Big)dx\notag\\
&\qquad+\nu\int s\dot{\mathbb{H}}^j\Big(\Delta \mathbb{H}_s^j+\nabla\cdot(\mathbf{u}\Delta \mathbb{H}^j)\Big)dx\notag\\
&\qquad+\int s\dot{\mathbf{u}}^j\Big[\big((\rot \mathbb{H})\times \mathbb{H}\big)_s^j
+\nabla\cdot\big[\mathbf{u}\cdot\big((\rot \mathbb{H})\times \mathbb{H}\big)^j\big]\Big]dx\notag\\
&\qquad-\int s\dot{\mathbb{H}}^j\nabla\cdot(\mathbf{u}\dot{\mathbb{H}}^j) dx+\int s\dot{\mathbb{H}}^j\Big((\mathbb{H}\cdot\nabla \mathbf{u})_s^j+\nabla\cdot\big(\mathbf{u}(\mathbb{H}\cdot\nabla \mathbf{u})^j\big)\Big)dx\notag\\
&\qquad-\int s\dot{\mathbb{H}}^j\Big((\mathbb{H}(\nabla\cdot \mathbf{u}))^j_s+\nabla\cdot(\mathbf{u}(\mathbb{H}\nabla\cdot \mathbf{u})^j)\Big)dx\notag\\
&\quad\stackrel{\triangle}{=}\frac{1}{2}\int (\varrho\vert\dot{\mathbf{u}}\vert^2
+\vert\dot{\mathbb{H}}\vert^2)dx+\sum_{i=1}^{8}\mathcal{A}_i.\label{(2.11**)}
\end{align}

From (2.12)-(2.15) in \cite[Lemma 2.2]{HJR}, we have get the estimates of $\mathcal{A}_{1}-\mathcal{A}_{3}$, so we only need to estimate $\mathcal{A}_{4}-\mathcal{A}_{8}$. From integration by parts, we find that there exists a positive constant $I$ such that
\begin{align}\label{G 3' estimate}
\mathcal{A}_{4}
&= -\nu\int s(\vert\nabla\dot{\mathbb{H}}\vert^2+\partial_l\dot{\mathbb{H}}^j\partial_k \mathbf{u}^k\partial_l\mathbb{H}^j
-\partial_k\dot{\mathbb{H}}^j\partial_k\mathbf{u}^l\partial_l\mathbb{H}^j
-\partial_k\dot{\mathbb{H}}^j\partial_l\mathbf{u}^k\partial_l{\mathbb{H}}^j)dx\notag\\
&\le -\frac{7\nu}{8}s\Vert\nabla\dot{\mathbb{\mathbb{H}}}\Vert_{L^2}^2
+Is(\Vert\nabla \mathbb{H}\Vert_{L^4}^4+\Vert\nabla \mathbf{u}\Vert_{L^4}^4).
\end{align}
For $\mathcal{A}_{5}$, we need to use $(\ref{modified derivatives})_1$,
\begin{align}\label{G 5' transform}
\mathcal{A}_{5}&=\int s\dot{\mathbf{u}}^j(\mathbb{H}\cdot\nabla\mathbb{H})^j_s
-s\dot{\mathbf{u}}^j\frac{1}{2}\nabla(\vert\mathbb{H}\vert^2)^j_s
-s\nabla\dot{\mathbf{u}}^j\mathbf{u}(\mathbb{H}\cdot\nabla\mathbb{H})^j
+s\nabla\dot{\mathbf{u}}^j\mathbf{u}\frac{1}{2}\nabla(\vert\mathbb{H}\vert^2)^jdx \notag\\
&\stackrel{\triangle}{=}\sum_{j=1}^{4}\mathcal{A}_{5,j},
\end{align}
after using integration by parts, (\ref{GN11}), (\ref{aaaa}) and (\ref{bbb}), we get
\begin{align}\label{G 6' estimate}
\mathcal{A}_{5,1}&=\int s\dot{\mathbf{u}}^j(\dot{{\mathbb{H}}}\cdot\nabla)\mathbb{H}^j
-s\dot{\mathbf{u}}^j(\mathbf{u}\cdot\nabla\mathbb{H})\cdot\nabla\mathbb{H}^j
-s\mathbb{H}\cdot\nabla\dot{\mathbf{u}}^j\dot{\mathbb{H}}^j
+s(\mathbb{H}\cdot\nabla)\dot{\mathbf{u}}^j(\mathbf{u}\cdot\nabla)\mathbb{H}^j dx \notag\\
&\le Is\Vert\nabla\dot{\mathbf{u}}\Vert_{L^2}\Vert\dot{\mathbb{H}}\Vert_{L^2}^{\frac{1}{2}}
\Vert\nabla\dot{\mathbb{H}}\Vert_{L^2}^{\frac{1}{2}}\Vert\nabla\mathbb{H}\Vert_{L^2}
+Is\Vert\nabla\dot{\mathbf{u}}\Vert_{L^2}\Vert\mathbb{H}\Vert_{L^2}^{\frac{1}{2}}
\Vert\nabla\mathbb{H}\Vert_{L^2}^{\frac{1}{2}}\Vert\nabla\mathbb{\dot{H}}\Vert_{L^2}\notag\\
&\quad+Is\Vert\nabla\dot{\mathbf{u}}\Vert_{L^2}\Vert\nabla\mathbf{u}\Vert_{L^2}
\Vert\nabla\mathbb{H}\Vert_{L^2}\Vert\nabla\mathbb{H}\Vert_{L^6}  \notag\\
&\le \frac{\mu}{32}s\Vert\nabla\dot{\mathbf{u}}\Vert_{L^2}^2
+\frac{\nu}{32}s\Vert\nabla\dot{\mathbb{H}}\Vert_{L^2}^2+I\Vert\dot{\mathbb{H}}\Vert_{L^2}^2+I\Vert\nabla \mathbb{H}\Vert_{L^6}^2,
\end{align}
similarly, we can deal with $\mathcal{A}_{5,2}-\mathcal{A}_{5,4}$ and also $\mathcal{A}_{6}-\mathcal{A}_{8}$. So
\begin{align}\label{G 4,5' estimate}
\sum_{i=5}^{8}\mathcal{A}_{i}&\le \frac{\mu}{8}s\Vert\nabla\dot{\mathbf{u}}\Vert_{L^2}^2
+\frac{\nu}{8}s\Vert\nabla\dot{\mathbb{H}}\Vert_{L^2}^2+I\Vert\dot{\mathbb{H}}\Vert_{L^2}^2+I\Big(\Vert\nabla \mathbf{u}\Vert_{L^6}^2+\Vert\nabla \mathbb{H}\Vert_{L^6}^2\Big).
\end{align}
And substituting $\mathcal{A}_{1}-\mathcal{A}_{8}$ into
(\ref{(2.11**)}), we get
\begin{align}\label{(2.17**)}
&\frac{d}{ds}\Big(\int s(\varrho\vert\dot{\mathbf{u}}\vert^2+\vert\dot{\mathbb{H}}\vert^2)dx\Big)
+s\Big(\Vert\nabla\dot{\mathbf{u}}\Vert_{L^2}^2+\Vert\nabla\dot{\mathbb{H}}\Vert_{L^2}^2\Big)\notag\\
&\quad\le I\Big(s\Vert\varrho^\frac{1}{2}\dot{\Phi}\Vert_{L^2}^2\Big)
+I\Big(\Vert\varrho^\frac{1}{2}\dot{\mathbf{u}}\Vert_{L^2}^2+\Vert\dot{\mathbb{H}}\Vert_{L^2}^2+\Vert\nabla \mathbf{u}\Vert_{L^2}^2\Big)\\
&\qquad+Is\Big(\Vert\nabla \mathbf{u}\Vert_{L^3}^2\Vert\nabla\Phi\Vert_{L^2}^2
+\Vert\nabla \mathbf{u}\Vert_{L^4}^4+\Vert\nabla \mathbb{H}\Vert_{L^4}^4\Big)+I\Big(\Vert\nabla \mathbf{u}\Vert_{L^6}^2+\Vert\nabla \mathbb{H}\Vert_{L^6}^2\Big).\notag
\end{align}

Secondly, multiplying $(\ref{MHD})_3$ by $s\dot{\Phi}$ and integrating the result over $\mathrm{R}^3$, we get
\begin{align}
\frac{\kappa}{2}\frac{d}{ds}\Big(s\Vert\nabla\Phi\Vert_{L^2}^2\Big)
&=\frac{\kappa}{2}\Vert\nabla\Phi\Vert_{L^2}^2-C_v\int s\varrho\vert\dot{\Phi}\vert^2dx\notag\\
&\quad+\lambda\int s(\nabla\cdot \mathbf{u})^2\dot{\Phi}dx
-\kappa\int s\nabla\Phi \cdot\nabla(\mathbf{u}\cdot\nabla\Phi)dx\notag\\
&\quad+2\mu\int s\vert\mathfrak{D}(\mathbf{u})\vert^2\dot{\Phi}dx-R\int s\varrho\Phi(\nabla\cdot \mathbf{u})\dot{\Phi}dx\notag\\
&\quad+\nu\int s\vert\rot \mathbb{H}\vert^2\dot{\Phi}dx\notag\\
&\stackrel{\triangle}{=}\frac{\kappa}{2}\Vert\nabla\Phi\Vert_{L^2}^2
-C_v\int s\varrho\vert\dot{\Phi}\vert^2dx+\sum_{i=1}^{5}\mathcal{B}_i.\label{(2.18**)}
\end{align}
From (2.22)-(2.24) in \cite[Lemma 2.2]{HJR}, it has gotten the estimates of $\mathcal{B}_1-\mathcal{B}_3$.
Using the $L^p$-estimates of elliptic equations, $(\ref{MHD})_3$ and (\ref{GN11}), we obtain
\begin{align}
\Vert\nabla^2\Phi\Vert_{L^2}
&\le I\Big(\Vert\varrho^\frac{1}{2}\dot{\Phi}\Vert_{L^2}+\Vert\Phi\nabla \mathbf{\mathbf{u}}\Vert_{L^2}
+\Vert\nabla \mathbf{\mathbf{u}}\Vert_{L^4}^2+\Vert\nabla \mathbb{H}\Vert_{L^4}^2\Big)\notag\\
&\le I\Big(\Vert\varrho^\frac{1}{2}\dot{\Phi}\Vert_{L^2}
+\Vert\nabla\Phi\Vert_{L^2}\Vert\nabla \mathbf{\mathbf{u}}\Vert_{L^3}
+\Vert\nabla \mathbf{\mathbf{u}}\Vert_{L^2}+\Vert\nabla \mathbf{\mathbf{u}}\Vert_{L^4}^2+\Vert\nabla \mathbb{H}\Vert_{L^4}^2\Big).\label{(2.20**)}
\end{align}
Analogous to (2.19) in \cite[Lemma 2.2]{HJR}, using (\ref{GN11}), (\ref{(2.4**)}) and (\ref{(2.5**)}) yields that
\begin{align}
\mathcal{B}_4
&\le Is\Vert\nabla \mathbf{\mathbf{u}}\Vert_{L^2}\Vert\nabla\Phi\Vert_{L^2}^\frac{1}{2}
\Vert\nabla^2\Phi\Vert_{L^2}^\frac{3}{2}\notag\\
&\le \frac{\mathbb{C}_v}{6}(s\Vert\varrho^\frac{1}{2}\dot{\Phi}\Vert_{L^2}^2)\label{(2.19**)}\\
&\quad+Is\Big(\Vert\nabla\Phi\Vert_{L^2}^2+\Vert\nabla\Phi\Vert_{L^2}^2\Vert\nabla \mathbf{\mathbf{u}}\Vert_{L^3}^2
+\Vert\nabla \mathbf{\mathbf{u}}\Vert_{L^2}^2+\Vert\nabla \mathbf{\mathbf{u}}\Vert_{L^4}^4+\Vert\nabla \mathbb{H}\Vert_{L^4}^4\Big).\notag
\end{align}
Similar to (2.21) in \cite[Lemma 2.2]{HJR}, from (\ref{GN11}) and (\ref{(2.5**)}), we can prove
\begin{align}
\int\vert\nabla \mathbb{H}\vert^2\Phi dx
&\le I\Big(\Vert\nabla \mathbb{H}\Vert_{L^2}^2+\Vert\nabla\Phi\Vert_{L^2}^2+\Vert\nabla \mathbb{H}\Vert_{L^6}^2\Big),\label{(2.21**)H}
\end{align}
further,
\begin{align}\label{H5' estimates}
\mathcal{B}_5
&=\nu\Big(\int s\vert\rot \mathbb{H}\vert^2\Phi\Big)_s-\nu\int\vert\rot \mathbb{H}\vert^2\Phi dx\notag\\
&\quad-\nu\int s\Phi\Big(2(\rot \mathbb{H})(\rot\dot{\mathbb{H}})
-2(\rot \mathbb{H})(\nabla\mathbf{u}^i\times\partial_i\mathbb{H})+\vert\rot \mathbb{H}\vert^2\nabla\cdot \mathbf{\mathbf{u}}\Big)dx\notag\\
&\le \nu\Big(\int s\vert\rot \mathbb{H}\vert^2\Phi\Big)_s
+\alpha_3s\Vert\nabla\dot{\mathbb{H}}\Vert_{L^2}^2+I\Big(\Vert\nabla \mathbb{H}\Vert_{L^2}^2+\Vert\nabla\Phi\Vert_{L^2}^2+\Vert\nabla \mathbb{H}\Vert_{L^6}^2\Big)\\
&\quad+I(\alpha_3)s\Big(\Vert\nabla\Phi\Vert_{L^2}^2(\Vert\nabla \mathbf{u}\Vert_{L^3}^2+\Vert\nabla \mathbb{H}\Vert_{L^3}^2)
+\Vert\nabla \mathbf{\mathbf{u}}\Vert_{L^2}^2+\Vert\nabla \mathbf{u}\Vert_{L^4}^4+\Vert\nabla \mathbb{H}\Vert_{L^4}^4\Big),\notag
\end{align}
for a small enough constant $\alpha_3\in(0,1)$.
Then, substituting $\mathcal{B}_1-\mathcal{B}_5$ into (\ref{(2.18**)}) with $\psi=s$ and $n=1$, we get that
\begin{align}\label{(2.25**)}
&\frac{d}{ds}\Big(s\Vert\nabla\Phi\Vert_{L^2}^2\Big)
+\Big(s\Vert\varrho^\frac{1}{2}\dot{\Phi}\Vert_{L^2}^2\Big)\notag\\
&\quad\le \frac{d}{ds}\int s\Big(\aleph(\nabla\cdot \mathbf{\mathbf{u}})^2+2\iota\vert\mathfrak{D}(\mathbf{u})\vert^2
+\nu\vert\rot \mathbb{H}\vert^2\Big)\Phi dx+\alpha_3s(\Vert\nabla\dot{\mathbf{u}}\Vert_{L^2}^2+\Vert\nabla\dot{\mathbb{H}}\Vert_{L^2}^2)\notag\\
&\qquad+I\Big(\Vert\nabla\Phi\Vert_{L^2}^2+\Vert\nabla \mathbf{u}\Vert_{L^2}^2+\Vert\nabla \mathbb{H}\Vert_{L^2}^2
+\Vert\nabla \mathbf{u}\Vert_{L^6}^2+\Vert\nabla \mathbb{H}\Vert_{L^6}^2\Big)\\
&\qquad+I(\alpha_3)s\Big(\Vert\nabla\Phi\Vert_{L^2}^2(\Vert\nabla \mathbf{u}\Vert_{L^3}^2+\Vert\nabla \mathbb{H}\Vert_{L^3}^2)+\Vert\nabla\Phi\Vert_{L^2}^2+\Vert\nabla \mathbf{u}\Vert_{L^2}^2+\Vert\nabla \mathbf{u}\Vert_{L^4}^4
+\Vert\nabla \mathbb{H}\Vert_{L^4}^4\Big).\notag
\end{align}

Next, combining (\ref{(2.17**)}) with $I^*$(\ref{(2.25**)})($I^*$ is a large enough number), after making $\alpha_3$ small suitably and
integrating the resulting equation over $(0,S)$, using (\ref{(2.5**)}) and (\ref{aaaa}) we have
\begin{align}
&\sup_{0\le s\le S}s\Big(\Vert\nabla\Phi\Vert_{L^2}^2+\Vert\varrho^{\frac{1}{2}}\dot{\mathbf{u}}\Vert_{L^2}^2
+\Vert\dot{\mathbb{H}}\Vert_{L^2}^2\Big)
+\int_0^Ss\Big(\Vert\varrho^\frac{1}{2}\dot{\Phi}\Vert_{L^2}^2+\Vert\nabla\dot{\mathbf{u}}\Vert_{L^2}^2
+\Vert\nabla\dot{\mathbb{H}}\Vert_{L^2}^2\Big)ds\notag\\
&\quad\le I+I\sup_{0\le s\le S}\int s\Big(\vert\nabla \mathbf{u}\vert^2+\vert\nabla \mathbb{H}\vert^2\Big)\Phi dx
+I\int_0^S\Big(\Vert\nabla \mathbf{u}\Vert_{L^6}^2+\Vert\nabla \mathbb{H}\Vert_{L^6}^2\Big)ds\notag\\
&\qquad+I\int_0^Ss\Big(\Vert\nabla\Phi\Vert_{L^2}^2(\Vert\nabla \mathbf{u}\Vert_{L^3}^2+\Vert\nabla \mathbb{H}\Vert_{L^3}^2)
+\Vert\nabla \mathbf{u}\Vert_{L^4}^4+\Vert\nabla \mathbb{H}\Vert_{L^4}^4\Big)ds\notag\\
&\quad\stackrel{\triangle}{=}I+\sum_{i=1}^{3}\mathcal{C}_i.\label{(2.26**)}
\end{align}
We need  to analysis the estimates of $\mathcal{C}_{1}-\mathcal{C}_{3}$. With  (22) in \cite{CFSA} and (\ref{(2.5**)}), we get
\begin{align}\label{(2.28**)}
&(\Vert\nabla\mathbf{u}\Vert_{L^6}+\Vert\nabla\mathbb{H}\Vert_{L^6})\notag\\
&\quad\le I\Big(E_0^{1/6}+\Vert\varrho^{\frac{1}{2}}\dot{\mathbf{u}}\Vert_{L^2}+\Vert\dot{\mathbb{H}}\Vert_{L^2}
+\Vert\nabla\Phi\Vert_{L^2}
+ +\Vert\nabla \mathbf{u} \Vert_{L^2}^3+\Vert\nabla \mathbb{H} \Vert_{L^2}^3 \Big)\notag\\
&\quad\le I\Big(1+\Vert\varrho^{\frac{1}{2}}\dot{\mathbf{u}}\Vert_{L^2}+\Vert\dot{\mathbb{H}}\Vert_{L^2}+\Vert\nabla\Phi\Vert_{L^2}
\Big),
\end{align}
with denoting that
$$\mathcal{F}\stackrel{\triangle}{=}1+\Vert\varrho^{\frac{1}{2}}\dot{\mathbf{u}}\Vert_{L^2}
+\Vert\dot{\mathbb{H}}\Vert_{L^2}
+\Vert\nabla\Phi\Vert_{L^2},$$
then, by using (\ref{(2.5**)}) and (\ref{aaaa}), we obtain
\begin{align}
\mathcal{C}_2
\le I\int_0^S\mathcal{F}^2ds\label{(2.29**)}
\le I.
\end{align}
Similar to (2.30) in \cite[Lemma 2.2]{HJR}, utilizing (\ref{(2.5**)}) and (\ref{(2.28**)}), one has
\begin{align}
\mathcal{C}_1
&\le I\Big(\Vert\nabla \mathbf{u}\Vert_{L^2}^2+\Vert\nabla \mathbb{H}\Vert_{L^2}^2\Big)
+Is\Vert\nabla\Phi\Vert_{L^2}\Big(\Vert\nabla \mathbf{u}\Vert_{L^6}^\frac{1}{2}
+\Vert\nabla \mathbb{H}\Vert_{L^6}^\frac{1}{2}\Big)\notag\\
&\le I+Is\Vert\nabla\Phi\Vert_{L^2}\mathcal{F}^\frac{1}{2}\label{(2.30**)}\\
&\le I+\frac{1}{8}s( \Vert\varrho^{\frac{1}{2}}\dot{\mathbf{u}}\Vert_{L^2}^2+\Vert\dot{\mathbb{H}}\Vert_{L^2}^2
+\Vert\nabla\Phi\Vert_{L^2}^2).\notag
\end{align}

Subsequently, exploiting Young's inequality, (\ref{GN11}), (\ref{(2.5**)}), (\ref{aaaa}) and (\ref{(2.28**)}),
we reach
\begin{align}\label{(2.31**)}
\mathcal{C}_3
&\le I\int_0^Ss\Big(\Vert\nabla\Phi\Vert_{L^2}^2(\Vert\nabla \mathbf{u}\Vert_{L^2}\Vert\nabla \mathbf{u}\Vert_{L^6}
+\Vert\nabla \mathbb{H}\Vert_{L^2}\Vert\nabla \mathbb{H}\Vert_{L^6})\Big)\notag\\
&\quad+I\int_0^Ss\Big(\Vert\nabla \mathbf{u}\Vert_{L^2}\Vert\nabla \mathbf{u}\Vert_{L^6}^3
+\Vert\nabla \mathbb{H}\Vert_{L^2}\Vert\nabla \mathbb{H}\Vert_{L^6}^3\Big)ds\notag\\
&\le I+I\int_0^Ss\mathcal{F}^3ds \notag\\
&\le I+\frac{1}{8}\sup_{0\le s\le S}s\Big(\Vert\nabla\Phi\Vert_{L^2}^2
+\Vert\varrho^{\frac{1}{2}}\dot{\mathbf{u}}\Vert_{L^2}^2+\Vert\dot{\mathbb{H}}\Vert_{L^2}^2\Big).
\end{align}

At last, substituting (\ref{(2.29**)})-(\ref{(2.31**)}) into (\ref{(2.26**)}) and integrating
(\ref{(2.20**)}) over $(0,S)$, considering them together can finish the verification of (\ref{(2.10**)}).

\end{proof}

\begin{Lemma}\label{L(2.3**)}
On the premise of Lemma $3.1$, $\exists$ a constant $I>0$ relying on $S$, such that the following estimate can be obtained:
\begin{align}
&\int_0^S\Big(\Vert\nabla\Phi\Vert_{L^m}^n+\Vert\varrho^\frac{1}{2}\dot{\mathbf{u}}\Vert_{L^m}^n
+\Vert\dot{\mathbb{H}}\Vert_{L^m}^n        +\Vert\Phi-1\Vert_{L^\infty}^n\notag\\
&\quad +\Vert \nabla\cdot \mathbf{u} \Vert_{L^\infty}^n     +\Vert\rot \mathbf{u}\Vert_{L^\infty}^n
+\Vert \nabla\cdot\mathbb{H} \Vert_{L^\infty}^n     +\Vert\rot \mathbb{H}\Vert_{L^\infty}^n\Big)ds\le I(S),\label{(2.32**)}
\end{align}
with the scopes:
\begin{align}
m\in(3,6)\quad and \quad 1<n<\frac{4m}{5m-6}<\frac{4}{3}.\label{(2.33**)}
\end{align}
\end{Lemma}

\begin{proof}
Because of (\ref{GN11}) and (\ref{(2.4**)}), we get
\begin{align*}
\Vert\dot{\mathbb{H}}\Vert_{L^m}^n\le I\Big(\Vert\dot{\mathbb{H}}\Vert_{L^2}^{\frac{n(6-m)}{2m}}
\Vert\nabla\dot{\mathbb{H}}\Vert_{L^2}^{\frac{n(3m-6)}{2m}}\Big),
\end{align*}
from  which and  (\ref{(2.10**)}), (\ref{(2.33**)}), we find
\begin{align}\label{Hlnm}
&\int_0^S\Vert\dot{\mathbb{H}}\Vert_{L^m}^nds\notag\\
&\quad\le I \sup_{0\le s\le S}  (s\Vert\dot{\mathbb{H}}\Vert_{L^2}^2)^{\frac{n(6-m)}{4m}}
\int_0^S s^{-\frac{n}{2}}(s\Vert\nabla\dot{\mathbb{H}}\Vert_{L^2}^2)^{\frac{n(3m-6)}{4m}}ds\notag\\
&\quad\le I(S)\Big(\int_0^S  s^{-\frac{2mn}{4m-3mn+6n}}\, ds\Big)^{\frac{4m+6n-3mn}{4m}}
\Big(    \int_0^S   \, s\Vert\nabla\dot{\mathbb{H}}\Vert_{L^2}^2 \,   ds\Big)^{\frac{n(3m-6)}{4m}}\notag\\
&\quad\le I(S).
\end{align}
Combine  \eqref{Hlnm} with (2.34) in \cite[Lemma 2.3]{HJR}, we achieve that
\begin{align}
\int_0^S\Big(\Vert\nabla\Phi\Vert_{L^m}^n+\Vert\varrho^\frac{1}{2}\dot{\mathbf{u}}\Vert_{L^m}^n
+\Vert\dot{\mathbb{H}}\Vert_{L^m}^n\Big)ds
\le I(S).\label{(2.34**)}
\end{align}

Utilizing (\ref{GN11}), (\ref{(2.4**)})-(\ref{aaaa}), (\ref{(2.28**)}),$(\ref{MHD})_4$ and Sobolev embedding inequality, we get
\begin{align}
\Vert\nabla\vert \mathbb{H}\vert^2\Vert_{L^m}
&\le I\Big(\Vert \mathbb{H}\Vert_{L^2}^\frac{1}{2}\Vert\nabla \mathbb{H}\Vert_{L^2}^\frac{1}{2}\Vert\nabla \mathbb{H}\Vert_{L^6}\Big)\le I\mathcal{F},\label{nabla H^2' L^m estimtes)}\\
\Vert\nabla^2\mathbb{H}\Vert_{L^m}
&\le I\Big(\Vert\dot{\mathbb{H}}\Vert_{L^m}
+\Vert \mathbb{H}\Vert_{L^2}^\frac{1}{2}\Vert\nabla \mathbb{H}\Vert_{L^2}^\frac{1}{2}\Vert\nabla \mathbf{u}\Vert_{L^6}\Big)\le I\Big(\Vert\dot{\mathbb{H}}\Vert_{L^m}+\mathcal{F}\Big),\label{nabla^2 H' L^m estimtes)}
\end{align}
coupled with $(\ref{MHD})_4$, $(\ref{modified derivatives})$, ($\ref{g,n estimates}$) and the conclusions of \cite[Lemma 2.3]{HJR}, we reach
\begin{align}
&\Vert\Phi-1\Vert_{L^\infty}+\Vert(\nabla\cdot \mathbf{u})\Vert_{L^\infty}
+\Vert\rot \mathbf{u}\Vert_{L^\infty}
+\Vert(\nabla\cdot\mathbb{H})\Vert_{L^\infty}+\Vert\rot \mathbb{H}\Vert_{L^\infty}\notag\\
&\quad\le
I\Big(1+\Vert(\nabla\Phi,\nabla\mathbb{G},\nabla\mathscr{N},\nabla(\vert \mathbb{H}\vert^2),\nabla^2 \mathbb{H})\Vert_{L^2}\notag\Big)\\
&\qquad+I\Big(\Vert(\nabla\Phi,\nabla\mathbb{G},\nabla\mathscr{N},\nabla(\vert \mathbb{H}\vert^2),\nabla^2 \mathbb{H})\Vert_{L^m}\Big)\notag\\
&\quad\le I\Big(\mathcal{F}+\Vert\nabla\Phi\Vert_{L^m}+\Vert\varrho^\frac{1}{2}\dot{\mathbf{u}}\Vert_{L^m}
+\Vert\dot{\mathbb{H}}\Vert_{L^m}\Big).\label{((theta-1,u,H)' estimtes)}
\end{align}

After integrating (\ref{((theta-1,u,H)' estimtes)}) over (0,S), combined with (\ref{(2.5**)}), (\ref{aaaa}) and (\ref{(2.34**)}), we will terminate the whole proof of (\ref{(2.32**)}).

\end{proof}

\begin{Lemma}\label{L(2.5**)}
Suppose that $(\varrho_0,\mathbf{u}_0,\Phi_0,\mathbb{H}_0)$ satisfy $(\ref{(1.17**)})$ and a constant $\zeta$, the following estimate can be obtained for $I_0\le\zeta$:
\begin{align}
&\sup_{0\le s\le S}\Big(\Vert\nabla\varrho\Vert_{L^2\cap L^m}+\Vert\varrho_s\Vert_{L^2}\Big)\notag\\
&\quad+\int_0^S\Big(\Vert\nabla^2\mathbf{u}\Vert_{L^m}^n+\Vert\nabla \mathbf{u}\Vert_{L^\infty}^n+\Vert\nabla^2\mathbb{H}\Vert_{L^m}^n
+\Vert\nabla \mathbb{H}\Vert_{L^\infty}^n\Big)ds
\le I(S),\label{(2.36**)}
\end{align}
where $(m,n)$ conform to $(\ref{(2.33**)})$.
\end{Lemma}

\begin{proof}
First of all, using $(\ref{MHD})_2$, (\ref{(2.4**)}), (\ref{nabla H^2' L^m estimtes)}), (\ref{nabla^2 H' L^m estimtes)}) and  the property of the elliptic operator $-\iota\Delta-(\iota+\aleph)\nabla(\nabla\cdot)$ (cf. \cite{YHHU}), it yields
\begin{align}
\Vert\nabla^2 \mathbf{u}\Vert_{L^m}
&\le I\Big(\Vert\varrho^\frac{1}{2}\dot{\mathbf{u}}\Vert_{L^m}+\Vert\nabla \mathfrak{P}(\varrho,\Phi)\Vert_{L^m}
+\Vert \mathbb{H}\cdot\nabla \mathbb{H}\Vert_{L^m}\Big)\notag\\
&\le I\Big(\Vert\varrho^\frac{1}{2}\dot{\mathbf{u}}\Vert_{L^m}+\Vert\nabla\Phi\Vert_{L^m}
+\Vert\Phi\Vert_{L^\infty}\Vert\nabla\varrho\Vert_{L^m}+\mathcal{F}^2\Big),\label{(2.38**)}\\
\Vert\nabla^2 \mathbb{H}\Vert_{L^m}&\le I\Big(\Vert\dot{\mathbb{H}}\Vert_{L^m}+\mathcal{F}^2\Big).\label{(2.38**)H}
\end{align}
for $m\in(3,6)$. Then, by (\ref{B-K-M inequality}), it yields
\begin{align}\label{(2.40**)}
\Vert\nabla \mathbf{u}\Vert_{L^\infty}&\le I\Big(\Vert\nabla \mathbf{u}\Vert_{L^2}+1\Big)+I\Big(\Vert(\nabla\cdot \mathbf{u})\Vert_{L^\infty}
+\Vert\rot \mathbf{u}\Vert_{L^\infty}\Big)\log\Big(e+\Vert\nabla\varrho\Vert_{L^m}\Big)\notag\\
&\quad+I\Big(\Vert(\nabla\cdot \mathbf{u})\Vert_{L^\infty}+\Vert\rot \mathbf{u}\Vert_{L^\infty}\Big)\times\notag\\
&\quad\log\Big(e+\Vert\varrho^\frac{1}{2}\dot{\mathbf{u}}\Vert_{L^m}+\Vert\nabla\Phi\Vert_{L^m}
+\Vert\Phi\Vert_{L^\infty}+\mathcal{F}^2\Big),
\end{align}
and
\begin{align}
\Vert\nabla \mathbb{H}\Vert_{L^\infty}&\le I\Big(\Vert\nabla \mathbb{H}\Vert_{L^2}+1\Big)\notag\\
&\quad+I\Big(\Vert(\nabla\cdot\mathbb{H})\Vert_{L^\infty}
+\Vert\rot \mathbb{H}\Vert_{L^\infty}\Big)\log\Big(e+\Vert\dot{\mathbb{H}}\Vert_{L^m}+\mathcal{F}^2\Big).\label{(2.40**)H}
\end{align}

Secondly, combining (2.37) in \cite[Lemma 2.5]{HJR} with (\ref{(2.38**)}), (\ref{(2.40**)}) and (\ref{(2.38**)}), one gets
\begin{align}
\frac{d}{ds}\log\Big(e+\Vert\nabla\varrho\Vert_{L^m}\Big)\le I\Psi(S)\log\Big(e+\Vert\nabla\varrho\Vert_{L^m}\Big),\label{(2.41**)}
\end{align}
where
\begin{align}
\Psi(S)&= I\Big(\Vert(\nabla\cdot \mathbf{u})\Vert_{L^\infty}+\Vert\rot \mathbf{u}\Vert_{L^\infty}\Big)
\log\Big(e+\Vert\varrho^\frac{1}{2}\dot{\mathbf{u}}\Vert_{L^m}+\Vert\nabla\Phi\Vert_{L^m}
+\Vert\Phi\Vert_{L^\infty}+\mathcal{F}^2\Big)\notag\\
&\quad+I\Big(\Vert\Phi\Vert_{L^\infty}+\Vert\varrho^\frac{1}{2}\dot{\mathbf{u}}\Vert_{L^m}+\Vert\nabla\Phi\Vert_{L^m}
+\mathcal{F}^2\Big).
\end{align}
Meanwhile, due to (\ref{(2.32**)}), it's easy to prove that
\begin{align*}
\int_0^S\Psi(S)ds\le I(S),
\end{align*}
which is combined with (\ref{(2.41**)}) yields
\begin{align}
\sup_{0\le s\le S}\Vert\nabla\varrho\Vert_{L^m}\le I(S)\quad m\in (3,6).\label{(2.42**)}
\end{align}

After following by (\ref{(2.38**)})-(\ref{(2.40**)H}), (\ref{(2.42**)}) and  Sobolev embedding inequality for $(m,n)$ as in (\ref{(2.33**)}), we have
\begin{align}
\int_0^S\Big(\Vert\nabla^2\mathbf{u}\Vert_{L^m}^n+\Vert\nabla \mathbf{u}\Vert_{L^\infty}^n+\Vert\nabla^2\mathbb{H}\Vert_{L^m}^n
+\Vert\nabla \mathbb{H}\Vert_{L^\infty}^n\Big)ds\le I(S).\label{(2.43**)}
\end{align}
Besides,   (2.37) in \cite[Lemma 2.5]{HJR} with $p=2$, and $\rm Gr\ddot{o}nwall$'s inequality  also deduce
\begin{align}
\sup_{0\le s\le S}\Vert\nabla\varrho\Vert_{L^2}\le I(S).\label{(2.44**)}
\end{align}

Finally, taking advantage of $(\ref{MHD})_1$ (\ref{GN11}), (\ref{(2.4**)}), (\ref{(2.5**)}) and (\ref{(2.42**)}), one reaches
\begin{align}
\Vert\varrho_s(s)\Vert_{L^2}\le I(S), \quad \forall s\in[0,S],
\end{align}
which is combined with (\ref{(2.42**)})-(\ref{(2.44**)}), we finish the proof of (\ref{(2.36**)}).

\end{proof}

\begin{Lemma}\label{L(2.6**)}
Suppose that $0< \varrho_1\le \inf_{x\in\mathrm{R}^3}\varrho_0(s)$, $\exists$ a constant $I(\varrho_1,S)$ relying on $\varrho_1$ and $S$, such that $\varrho(x,s)$ satisfies
\begin{align}
\varrho(x,s)\ge I(\varrho_1,S), \quad \forall (x,s)\in(\mathrm{R}^3;[0,S]),\label{(2.45**)}
\end{align}
and the following estimate can be obtained:
\begin{align}
&\sup_{0\le s\le S}s\Big(\Vert\nabla^2 \mathbf{u}\Vert_{L^2}^2+\Vert\nabla^2 \mathbb{H}\Vert_{L^2}^2
+\Vert \mathbf{u}_s\Vert_{L^2}^2+\Vert \mathbb{H}_s\Vert_{L^2}^2\Big)\notag\\
&\qquad+\int_0^S\Big(\Vert\nabla^2 \mathbf{u}\Vert_{L^2}^2+\Vert\nabla^2 \mathbb{H}\Vert_{L^2}^2
+\Vert \mathbf{u}_s\Vert_{L^2}^2+\Vert \mathbb{H}_s\Vert_{L^2}^2
+s(\Vert\nabla \mathbf{u}_s\Vert_{L^2}^2+\Vert\nabla \mathbb{H}_s\Vert_{L^2}^2)\Big)ds\notag\\
&\quad\le I(S).\label{(2.46**)}
\end{align}

\end{Lemma}

\begin{proof}
Similar to (2.45) in \cite[Lemma 2.6]{HJR}, from $(\ref{MHD})_1$ and (\ref{(2.32**)}), we can prove (\ref{(2.45**)}) instantly.
Then, we consider about (\ref{(2.46**)}). Using the following fact
\begin{align}
(\partial_j\mathfrak{P})_s+\partial_k(\mathbf{u}^k\partial_j\mathfrak{P})
=\partial_j(\mathcal{R}\varrho\Phi)-\partial_k(\mathfrak{P}\partial^k),
\end{align}
which is considered with $(\ref{MHD})_2$, $(\ref{MHD})_4$, (\ref{GN11}) and (\ref{(2.36**)}), we get
\begin{align}
&\Big(\Vert\nabla^2 \mathbf{u}\Vert_{L^2}+\Vert\nabla^2 \mathbb{H}\Vert_{L^2}\Big)\notag\\
&\quad\le I\Big(1+\Vert\dot{\mathbf{u}}\Vert_{L^2}+\Vert\nabla\Phi\Vert_{L^2}
+\Vert \mathbb{H}\cdot\nabla \mathbb{H}\Vert_{L^2}+\Vert\dot{\mathbb{H}}\Vert_{L^2}+\Vert \mathbb{H}\cdot\nabla \mathbf{u}\Vert_{L^2}\Big)\notag\\
&\quad\le I\Big(1+\Vert\dot{\mathbf{u}}\Vert_{L^2}+\Vert\dot{\mathbb{H}}\Vert_{L^2}
+\Vert\nabla\Phi\Vert_{L^2}\Big),\label{(nabla^2 u,H' estimates}
\end{align}
Coupled with (\ref{(2.5**)}), (\ref{aaaa}) (\ref{(2.10**)}) and (\ref{(2.45**)}),
\begin{align}
&\sup_{0\le s\le S}s\Big(\Vert\nabla^2 \mathbf{u}(s)\Vert_{L^2}^2+\Vert\nabla^2 \mathbb{H}(s)\Vert_{L^2}^2\Big)
+\int_0^S\Big(\Vert\nabla^2 \mathbf{u}\Vert_{L^2}^2+\Vert\nabla^2 \mathbb{H}\Vert_{L^2}^2\Big)ds\le I(S).\label{(2.47**)}
\end{align}

In addition, from (\ref{GN11}) and the definition of $(\ref{modified derivatives})_1$:
$ \mathscr{L}_s\stackrel{\triangle}{=} \dot{\mathscr{L}}-\mathbf{u}\cdot\nabla \mathscr{L}$, it has
\begin{align}
\Vert\nabla \mathbf{u}_s\Vert_{L^2}^2
&\le I(\Vert\nabla \dot{\mathbf{u}}\Vert_{L^2}^2+\Vert\mathbf{u}\cdot\nabla\mathbf{u}\Vert_{L^2}^2)\notag\\
&\le I(\Vert\nabla \dot{\mathbf{u}}\Vert_{L^2}^2+\Vert\mathbf{u}\Vert_{L^2}
\Vert\nabla\mathbf{u}\Vert_{L^2}\Vert\nabla^2 \mathbf{u}\Vert_{L^2}^2),
\end{align}

after taking advantage of (\ref{(2.5**)}), (\ref{aaaa}), (\ref{(2.10**)}) and (\ref{(2.47**)}), it obtains
\begin{align}
&\sup_{0\le s\le S}s\Big(\Vert \mathbf{u}_s\Vert_{L^2}^2+\Vert \mathbb{H}_s\Vert_{L^2}^2\Big)
+\int_0^S\Big(\Vert \mathbf{u}_s\Vert_{L^2}^2+\Vert \mathbb{H}_s\Vert_{L^2}^2\Big)ds\le I(S).\label{(2.48**)}
\end{align}

In the end, using (\ref{GN11}), (\ref{(2.5**)}), (\ref{(2.10**)}) and (\ref{(2.47**)}), it yields that
\begin{align}\label{(2.49**)}
&\int_0^Ss\Big(\Vert\nabla \mathbf{u}_s\Vert_{L^2}^2+\Vert\nabla \mathbb{H}_s\Vert_{L^2}^2\Big)ds\notag\\
&\quad\le I\int_0^Ss(\Vert\nabla\dot{\mathbf{u}}\Vert_{L^2}^2
+\Vert\nabla(\mathbf{u}\cdot\nabla\mathbf{u})\Vert_{L^2}^2
+\Vert\nabla\dot{\mathbb{H}}\Vert_{L^2}^2
+\Vert\nabla(\mathbf{u}\cdot\nabla\mathbb{H})\Vert_{L^2}^2)ds \notag\\
&\quad\le I+I\int_0^Ss(\Vert\nabla\mathbf{u}\Vert_{L^2}\Vert\nabla^2\mathbf{u}\Vert_{L^2}^3
+\Vert\nabla\mathbf{u}\Vert_{L^2}\Vert\nabla^2\mathbf{u}\Vert_{L^2}\Vert\nabla^2\mathbb{H}\Vert_{L^2}^2) ds\notag\\
&\quad\le I.
\end{align}
Putting (\ref{(2.47**)})-(\ref{(2.49**)}) together, we can complete the whole verification of (\ref{(2.46**)}).

\end{proof}

\begin{Lemma}\label{Lbbb}
On the premise of Lemma $\ref{Main result}$, the following estimate can be obtained:
\begin{align}
&\sup_{0\le s\le S}\Big(\Vert\nabla\Phi\Vert_{L^2}^2+\Vert\nabla \mathbf{u}\Vert_{L^3}^3
+\Vert\nabla \mathbb{H}\Vert_{L^3}^3\Big)\notag\\
&\quad+\int_0^S\Big(\Vert\dot{\Phi}\Vert_{L^2}^2+\Vert\nabla^2\Phi\Vert_{L^2}^2
+\Vert\Phi_s\Vert_{L^2}^2\Big)ds\le I(S).\label{(2.50**)}
\end{align}

\end{Lemma}

\begin{proof}
Initially, we introduce some serviceable conclusions from \cite[Lemma 2.7]{HJR}:
\begin{align}\label{(2.58**)}
\Vert\mathscr{L}\Vert_{L^6}^2\le I\Vert\mathscr{L}\Vert_{L^3}^\frac{1}{2}
\Vert\vert\mathscr{L}\vert^\frac{1}{2}\nabla\mathscr{L}\Vert_{L^2}.
\end{align}
and
\begin{align}\label{(2.73**)}
\Vert\nabla \mathbf{u}\Vert_{L^{\frac{9m}{4m-9}}}^3\le I\Big(1+\Vert\nabla \mathbf{u}\Vert_{H^1}^2\Big)
\Big(1+\Vert\nabla \mathbf{u}\Vert_{L^3}^2\Big), \quad \frac{9}{2}\le m<6.
\end{align}

Second step, multiplying $\nabla\cdot(\ref{MHD})_2$ by $\vert(\nabla\cdot \mathbf{u})\vert(\nabla\cdot \mathbf{u})$ and $\nabla\cdot(\ref{MHD})_4$ by $\vert(\nabla\cdot\mathbb{H})\vert(\nabla\cdot\mathbb{H})$, adding up and integrating the result over $\mathrm{R}^3$, it yields
\begin{align}
&\frac{1}{3}\frac{d}{ds}\int\Big(\varrho\vert(\nabla\cdot \mathbf{u})\vert^3+\vert(\nabla\cdot\mathbb{H})\vert^3\Big)dx+(2\iota+\aleph)\int\Big(\vert(\nabla\cdot \mathbf{u})\vert(\vert\nabla\vert(\nabla\cdot \mathbf{u})\vert\vert^2
+\vert\nabla(\nabla\cdot \mathbf{u})\vert^2)\Big)dx\notag\\
&\qquad+\nu\int\Big(\vert(\nabla\cdot\mathbb{H})\vert(\vert\nabla\vert(\nabla\cdot\mathbb{H})\vert\vert^2
+\vert\nabla(\nabla\cdot\mathbb{H})\vert^2)\Big)dx\notag\\
&\quad=-\int(\mathbf{u}_s\cdot\nabla\varrho)(\vert(\nabla\cdot \mathbf{u})\vert(\nabla\cdot \mathbf{u}))dx
+\int\nabla \mathfrak{P}(\varrho,\Phi)\cdot\nabla(\vert(\nabla\cdot \mathbf{u})\vert(\nabla\cdot \mathbf{u}))dx\notag\\
&\qquad-\int\partial_j(\varrho \mathbf{u}^i)\partial_i\mathbf{u}^j(\vert(\nabla\cdot \mathbf{u})\vert(\nabla\cdot \mathbf{u}))dx
-\int\Big[(\rot \mathbb{H})\times \mathbb{H}\Big]\cdot\nabla(\vert(\nabla\cdot \mathbf{u})\vert(\nabla\cdot \mathbf{u}))dx\notag\\
&\qquad+\int\nabla\cdot(\mathbb{H}\cdot\nabla \mathbf{u})(\vert(\nabla\cdot\mathbb{H})\vert(\nabla\cdot\mathbb{H}))dx
-\int\nabla\cdot(\mathbb{H}(\nabla\cdot \mathbf{u}))(\vert(\nabla\cdot\mathbb{H})\vert(\nabla\cdot\mathbb{H}))dx\notag\\
&\qquad-\int\nabla\cdot(\mathbf{u}\cdot\nabla \mathbb{H})(\vert(\nabla\cdot\mathbb{H})\vert(\nabla\cdot\mathbb{H}))dx\notag\\
&\quad\stackrel{\triangle}{=}\sum_{i=1}^{7}\mathcal{D}_i.\label{(2.53**)}
\end{align}
Because of $(\ref{MHD})_2$, yields that
\begin{align}
\mathcal{D}_1&=-\int\varrho^{-1}\Big(\iota\Delta \mathbf{u}+(\iota+\aleph)\nabla(\nabla\cdot \mathbf{u})
-\nabla \mathfrak{P}(\varrho,\Phi)-\varrho \mathbf{u}\cdot\nabla \mathbf{u}\Big)\cdot\nabla\varrho(\vert(\nabla\cdot \mathbf{u})\vert(\nabla\cdot \mathbf{u}))dx\notag\\
&\quad-\int\varrho^{-1}\Big[(\rot \mathbb{H})\times \mathbb{H}\Big]\cdot\nabla\varrho(\vert(\nabla\cdot \mathbf{u})\vert(\nabla\cdot \mathbf{u}))dx\notag\\
&\stackrel{\triangle}{=}\sum_{j=1}^{2}\mathcal{D}_{1,j},\label{(2.55**)}
\end{align}
we have the estimates of $\mathcal{D}_{1,1}$ and $\mathcal{D}_{2}-\mathcal{D}_{3}$ from (2.59)-(2.60) in \cite[Lemma 2.7]{HJR}, then we analysis
$\mathcal{D}_{1,2}$:
\begin{align}
\mathcal{D}_{1,2}
&\le I\Vert\nabla\varrho\Vert_{L^3}\Vert \mathbb{H}\Vert_{L^6}\Vert\nabla \mathbb{H}\Vert_{L^6}\Vert(\nabla\cdot \mathbf{u})\Vert_{L^6}^2\notag\\
&\le I\Vert\nabla \mathbb{H}\Vert_{H^1}\Vert(\nabla\cdot \mathbf{u})\Vert_{L^3}^\frac{1}{2}
\Vert\vert(\nabla\cdot \mathbf{u})\vert^{\frac{1}{2}}(\nabla\cdot \mathbf{u})\Vert_{L^2}\label{(2.55**)H}\\
&\le \frac{2\iota+\aleph}{8}\Vert\vert(\nabla\cdot \mathbf{u})\vert^\frac{1}{2}(\nabla\cdot \mathbf{u})\Vert_{L^2}^2
+I\Vert\nabla \mathbb{H}\Vert_{H^1}^2\Big(1+\Vert(\nabla\cdot \mathbf{u})\Vert_{L^3}^2\Big),\notag
\end{align}
from (\ref{GN11}), (\ref{(2.5**)}), (\ref{(2.36**)}) and (\ref{(2.58**)}). Sequently, due to (\ref{GN11}), (\ref{(2.5**)}) and integration by parts, we have
\begin{align}
\sum_{i=4}^{7}\mathcal{D}_{i}
&\le \frac{2\iota+\aleph}{8}\Vert\vert(\nabla\cdot \mathbf{u})\vert^\frac{1}{2}(\nabla\cdot \mathbf{u})\Vert_{L^2}^2
+\frac{\nu}{8}\Vert\vert(\nabla\cdot\mathbb{H})\vert^\frac{1}{2}(\nabla\cdot\mathbb{H})\Vert_{L^2}^2\notag\\
&\quad+I\Vert\nabla^2\Phi\Vert_{L^2}^2+I\Big(\Vert\nabla \mathbf{u}\Vert_{H^1}^2+\Vert\nabla \mathbf{u}\Vert_{L^{\frac{9m}{4m-9}}}^3\Big)\label{J4,5,6,7' estimates}\\
&\quad+I\Big(1+\Vert\nabla\Phi\Vert_{L^2}^2+\Vert\nabla \mathbf{u}\Vert_{H^1}^2+\Vert\nabla \mathbb{H}\Vert_{H^1}^2\Big)
\Big(1+\Vert(\nabla\cdot \mathbf{u})\Vert_{L^3}^2+\Vert(\nabla\cdot\mathbb{H})\Vert_{L^3}^2\Big).\notag
\end{align}
After substituting $\mathcal{D}_1-\mathcal{D}_7$ into (\ref{(2.53**)}), we testify that
\begin{align}
&\frac{d}{ds}\int\Big(\varrho\vert(\nabla\cdot \mathbf{u})\vert^3+\vert(\nabla\cdot\mathbb{H})\vert^3\Big)dx
+\Vert\vert(\nabla\cdot \mathbf{u})\vert^{\frac{1}{2}}\nabla(\nabla\cdot \mathbf{u})\Vert_{L^2}^2
+\Vert\vert(\nabla\cdot\mathbb{H})\vert^{\frac{1}{2}}\nabla(\nabla\cdot\mathbb{H})\Vert_{L^2}^2\notag\\
&\quad\le I\Vert\nabla^2\Phi\Vert_{L^2}^2+I\Big(\Vert\nabla \mathbf{u}\Vert_{H^1}^2+\Vert\nabla \mathbf{u}\Vert_{L^{\frac{9m}{4m-9}}}^3\Big)\notag\\
&\qquad+I\Big(1+\Vert\nabla\Phi\Vert_{L^2}^2+\Vert\nabla \mathbf{u}\Vert_{H^1}^2+\Vert\nabla \mathbb{H}\Vert_{H^1}^2\Big)
\Big(1+\Vert(\nabla\cdot \mathbf{u})\Vert_{L^3}^2+(\nabla\cdot\mathbb{H})\Vert_{L^3}^2\Big).\label{(2.62**)}
\end{align}

Thirdly, multiplying $\rot(\ref{MHD})_2$ by $\vert\rot \mathbf{u}\vert\rot \mathbf{u}$ and $\rot(\ref{MHD})_4$ by $\vert\rot \mathbb{H}\vert\rot \mathbb{H}$, adding up and integrating the result over $\mathrm{R}^3$, one reaches
\begin{align}
&\frac{1}{3}\frac{d}{ds}\int\Big(\varrho\vert\rot \mathbf{u}\vert^3+\vert\rot  \mathbb{H}\vert^3\Big)dx+\iota\int\Big(\vert\rot \mathbf{u}\vert\big(\vert\nabla(\vert\rot \mathbf{u}\vert)\vert^2
+\vert\nabla(\rot  \mathbf{u})\vert^2\big)\Big)dx\notag\\
&\qquad+\nu\int\Big(\vert\rot \mathbb{H}\vert\big(\vert\nabla(\vert\rot \mathbb{H}\vert)\vert^2
+\vert\nabla(\rot \mathbb{H})\vert^2\big)\Big)dx\notag\\
&\quad=-\int(\nabla\varrho\times \mathbf{u}_s)\cdot(\vert\rot \mathbf{u}\vert\rot \mathbf{u})dx
-\int\nabla(\varrho \mathbf{u}^i)\times(\partial_i \mathbf{u})\cdot(\vert\rot \mathbf{u}\vert\rot \mathbf{u})dx\notag\\
&\qquad+\int\rot(\mathbb{H}\cdot\nabla \mathbb{H})\cdot(\vert\rot \mathbf{u}\vert\rot \mathbf{u})dx
+\int\rot(\mathbb{H}\cdot\nabla \mathbf{u})\cdot(\vert\rot \mathbb{H}\vert\rot \mathbb{H})dx\notag\\
&\qquad-\int\rot(\mathbb{H}(\nabla\cdot \mathbf{u}))\cdot(\vert\rot \mathbb{H}\vert\rot \mathbb{H})dx
-\int\rot(\mathbf{u}\cdot\nabla \mathbb{H})\cdot(\vert\rot \mathbb{H}\vert\rot \mathbb{H})dx\notag\\
&\quad\stackrel{\triangle}{=}\sum_{i=1}^{6}\mathcal{E}_i.\label{(2.63**)}
\end{align}
For $\mathcal{E}_1$, in terms of (2.64) in \cite[Lemma 2.7]{HJR}, $(\ref{MHD})_2$ and (\ref{(1.27**)}), we have
\begin{align}
\mathcal{E}_1&=(2\iota+\aleph)\int(\nabla\cdot \mathbf{u})(\nabla\log\varrho)\cdot
(\rot(\vert\rot \mathbf{u}\vert\rot \mathbf{u}))dx\notag\\
&\quad+\int(\nabla\log\varrho)\times\Big[\iota\rot(\rot \mathbf{u})+\nabla \mathfrak{P}(\varrho,\Phi)+\varrho \mathbf{u}\cdot\nabla \mathbf{u}\Big]
\cdot(\vert\rot \mathbf{u}\vert\rot \mathbf{u})dx\notag\\
&\quad+\int(\nabla\log\varrho)\times\Big[(\rot \mathbb{H})\times \mathbb{H}\Big]\cdot
(\vert\rot \mathbf{u}\vert\rot \mathbf{u})dx\notag\\
&\stackrel{\triangle}{=}\sum_{j=1}^{3}\mathcal{E}_{1,j}.\label{(2.64**)}
\end{align}
due to (\ref{GN11}), (\ref{(2.5**)}), (\ref{(2.36**)}) and (\ref{(2.58**)}), we obtain
\begin{align}
\mathcal{E}_{1,3}
&\le I\Vert\nabla\varrho\Vert_{L^3}\Vert \mathbb{H}\Vert_{L^6}\Vert\nabla \mathbb{H}\Vert_{L^6}\Vert\rot \mathbf{u}\Vert_{L^6}^2\notag\\
&\le \frac{\iota}{8}\Vert\vert\rot \mathbf{u}\vert^{\frac{1}{2}}\nabla(\rot \mathbf{u})\Vert_{L^2}^2
+I\Vert\nabla \mathbb{H}\Vert_{H^1}^2(1+\Vert\rot \mathbf{u}\Vert_{L^3}^2),\label{(2.64**)H}
\end{align}
Analogous to (\ref{J4,5,6,7' estimates}), using (\ref{GN11}) (\ref{(2.5**)}) and (\ref{aaaa}), we reach
\begin{align}
\sum_{i=3}^{6}\mathcal{E}_{i}
&\le\frac{\iota}{8}\Vert\vert\rot \mathbf{u}\vert^{\frac{1}{2}}\nabla(\rot \mathbf{u})\Vert_{L^2}^2+\frac{\nu}{8}\Vert\vert\rot \mathbb{H}\vert^{\frac{1}{2}}\nabla(\rot \mathbb{H})\Vert_{L^2}^2\notag\\
&\quad+I\Big(1+\Vert\nabla \mathbf{u}\Vert_{H^1}^2+\Vert\nabla \mathbb{H}\Vert_{H^1}^2\Big)
\Big(1+\Vert\rot \mathbf{u}\Vert_{L^3}^2+\Vert\rot \mathbb{H}\Vert_{L^3}^2\Big).\label{K4,5,6' estimates}
\end{align}
Substituting (\ref{(2.64**)H})-(\ref{K4,5,6' estimates}) into (\ref{(2.63**)}) with
(2.66) in \cite[Lemma 2.7]{HJR}, we reach
\begin{align}\label{(2.66**)}
&\frac{d}{ds}\int\Big(\varrho\vert\rot \mathbf{u}\vert^3+\vert\rot  \mathbb{H}\vert^3\Big)dx
+\Vert\vert\rot \mathbf{u}\vert^{\frac{1}{2}}\nabla(\rot \mathbf{u})\Vert_{L^2}^2
+\Vert\vert\rot \mathbb{H}\vert^{\frac{1}{2}}\nabla(\rot \mathbb{H})\Vert_{L^2}^2\notag\\
&\quad\le I\Vert\nabla^2\Phi\Vert_{L^2}^2+I\Big(\Vert\nabla \mathbf{u}\Vert_{H^1}^2+\Vert\nabla \mathbf{u}\Vert_{L^{\frac{9m}{4m-9}}}^3\Big)\notag\\
&\qquad+I\Big(1+\Vert\nabla\Phi\Vert_{L^2}^2+\Vert\nabla \mathbf{u}\Vert_{H^1}^2+\Vert\nabla \mathbb{H}\Vert_{H^1}^2\Big)
\Big(1+\Vert\rot \mathbf{u}\Vert_{L^3}^2+\Vert\rot \mathbb{H}\Vert_{L^3}^2\Big).
\end{align}

Next step, multiplying $(\ref{MHD})_3$ by $\dot{\Phi}$ yields that
\begin{align}
&\frac{\kappa}{2}\frac{d}{ds}\Big(\Vert\nabla\Phi\Vert_{L^2}^2\Big)
+\mathbb{C}_v\int\varrho\vert\dot{\Phi}\vert^2dx\notag\\
&\quad=\aleph\int(\nabla\cdot \mathbf{u})^2\dot{\Phi}dx
-\kappa\int\nabla\Phi \cdot\nabla(\mathbf{u}\cdot\nabla\Phi)dx
+2\iota\int\vert\mathfrak{D}(\mathbf{u})\vert^2\dot{\Phi}dx\notag\\
&\qquad-\mathcal{R}\int \varrho\Phi(\nabla\cdot \mathbf{u})\dot{\Phi}dx+\nu\int\vert\rot \mathbb{H}\vert^2\dot{\Phi}dx\notag\\
&\quad\stackrel{\triangle}{=}\sum_{i=1}^{5}\mathcal{E}_i.\label{(2.67**)}
\end{align}
Because of (\ref{(2.4**)}), (\ref{(2.5**)}), (\ref{GN11}) and $(\ref{MHD})_3$, we can deduce that
\begin{align}
\Vert\nabla^2\Phi\Vert_{L^2}
\le I\Big(1+\Vert\dot{\Phi}\Vert_{L^2}
+(\Vert\nabla\Phi\Vert_{L^2}+\Vert\nabla \mathbf{u}\Vert_{L^6})\Vert\nabla \mathbf{u}\Vert_{L^3}
+\Vert\nabla \mathbb{H}\Vert_{L^3}\Vert\nabla \mathbb{H}\Vert_{L^6}\Big).\label{(2.68**)}
\end{align}
According to (2.69) and (2.70) in \cite[Lemma 2.7]{HJR}, using (\ref{GN11}), (\ref{(2.5**)}) and (\ref{(2.68**)}) obtains
\begin{align*}
\sum_{i=1}^{5}\mathcal{E}_i
&\le I\Vert\nabla\Phi\Vert_{L^2}\Vert\nabla^2\Phi\Vert_{L^2}\Vert\nabla \mathbf{u}\Vert_{L^3}\\
&\quad+I\Vert\dot{\Phi}\Vert_{L^2}\Big(\Vert\nabla \mathbf{u}\Vert_{L^4}^2+\Vert\Phi-1\Vert_{L^6}\Vert\nabla \mathbf{u}\Vert_{L^3}
+\Vert\nabla \mathbf{u}\Vert_{L^2}+\Vert\nabla \mathbb{H}\Vert_{L^4}^2\Big)\\
&\le \alpha_4\Vert\dot{\Phi}\Vert_{L^2}^2\\
&\quad+I(\alpha_4)\Big(1+\Vert\nabla\Phi\Vert_{L^2}^2+\Vert\nabla \mathbf{u}\Vert_{H^1}^2+\Vert\nabla \mathbb{H}\Vert_{H^1}^2\Big)
\Big(1+\Vert\nabla \mathbf{u}\Vert_{L^3}^2+\Vert\nabla \mathbb{H}\Vert_{L^3}^2\Big),
\end{align*}
which is combined with (\ref{(2.45**)}), (\ref{(2.67**)}) and (\ref{(2.68**)}), further we can prove that
\begin{align}
&\frac{d}{ds}\Vert\nabla\Phi\Vert_{L^2}^2+\Vert\dot{\Phi}\Vert_{L^2}^2+\Vert\nabla^2\Phi\Vert_{L^2}^2\notag\\
&\quad\le I\Big(1+\Vert\nabla\Phi\Vert_{L^2}^2+\Vert\nabla \mathbf{u}\Vert_{H^1}^2+\Vert\nabla \mathbb{H}\Vert_{H^1}^2\Big)
\Big(1+\Vert\nabla \mathbf{u}\Vert_{L^3}^2+\Vert\nabla \mathbb{H}\Vert_{L^3}^2\Big),\label{(2.71**)}
\end{align}
for suitably small $\alpha_4>0$.

Eventually, adding (\ref{(2.62**)}) and (\ref{(2.66**)}) to (\ref{(2.71**)}),
coupled with (\ref{(2.45**)}) and (\ref{(2.73**)}), we can certify
\begin{align}\label{(2.72**)}
&\frac{d}{ds}\Big(\Vert\nabla\Phi\Vert_{L^2}^2+\Vert(\nabla\cdot \mathbf{u})\Vert_{L^3}^3+\Vert(\nabla\cdot\mathbb{H})\Vert_{L^3}^3
+\Vert\rot \mathbf{u}\Vert_{L^3}^3+\Vert\rot \mathbb{H}\Vert_{L^3}^3\Big)+\Vert\dot{\Phi}\Vert_{L^2}^2
+\Vert\nabla^2\Phi\Vert_{L^2}^2\notag\\
&\qquad+\Big(\Vert\vert(\nabla\cdot \mathbf{u})\vert^{\frac{1}{2}}\nabla(\nabla\cdot \mathbf{u})\Vert_{L^2}^2
+\Vert\vert(\nabla\cdot\mathbb{H})\vert^{\frac{1}{2}}\nabla((\nabla\cdot\mathbb{H}))\Vert_{L^2}^2\Big)\notag\\
&\qquad+\Big(\Vert\vert\rot \mathbf{u}\vert^{\frac{1}{2}}\nabla(\rot \mathbf{u})\Vert_{L^2}^2
+\Vert\vert\rot \mathbb{H}\vert^{\frac{1}{2}}\nabla(\rot \mathbb{H})\Vert_{L^2}^2\Big)\notag\\
&\quad\le I\Big(1+\Vert\nabla \mathbf{u}\Vert_{H^1}^2\Big)\Big(1+\Vert\nabla \mathbf{u}\Vert_{L^3}^2\Big)\notag\\
&\qquad+I\Big(1+\Vert\nabla\Phi\Vert_{L^2}^2+\Vert\nabla \mathbf{u}\Vert_{H^1}^2+\Vert\nabla \mathbb{H}\Vert_{H^1}^2\Big)
\Big(1+\Vert\nabla \mathbf{u}\Vert_{L^3}^2+\Vert\nabla \mathbb{H}\Vert_{L^3}^2\Big).
\end{align}
After integrating the result over $(0,S)$, using (\ref{aaaa}), (\ref{(2.46**)}) and Gronwall' inequality, we obtain a available estimate,
\begin{align}
&\sup_{0\le s\le S}\Big(\Vert\nabla\Phi\Vert_{L^2}^2+\Vert\nabla \mathbf{u}\Vert_{L^3}^3
+\Vert\nabla \mathbb{H}\Vert_{L^3}^3\Big)
+\int_0^S\Big(\Vert\dot{\Phi}\Vert_{L^2}^2+\Vert\nabla^2\Phi\Vert_{L^2}^2\Big)ds\le I(S),
\end{align}
which is combined with the following conclusion
\begin{align}\label{theta_s' estimates}
\Vert \Phi_s\Vert_{L^2}\le I\Big(\Vert\dot{\Phi}\Vert_{L^2}+\Vert \mathbf{u}\Vert_{L^\infty}\Vert\nabla \Phi\Vert_{L^2}\Big)
\le I\Big(\Vert\dot{\Phi}\Vert_{L^2}+\Vert\nabla \mathbf{u}\Vert_{H^1}\Big)\in L^2(0,S),
\end{align}
we can finish the proof of (\ref{(2.50**)}).

\end{proof}

\begin{Lemma}\label{L(2.8**)}
On the premise of Lemma $\ref{Main result}$, the following estimate can be obtained:
\begin{align}
\Phi(x,s)\ge\inf_{x\in\mathrm{R}^3}\Phi_0(x)\exp\Big\{-\frac{\mathcal{R}}{\mathbb{C}_v}\int_0^S\Vert(\nabla\cdot \mathbf{u})\Vert_{L^{\infty}}ds\Big\}>0,\label{(2.74**)}
\end{align}
for $\forall(x,s)\in\mathrm{R}^3\times[0,S]$, and
\begin{align}
&\sup_{0\le s\le S}s\Big(\Vert\dot{\Phi}\Vert_{L^2}^2+\Vert\nabla^2\Phi\Vert_{L^2}^2
+\Vert\Phi_s\Vert_{L^2}^2\Big)
+\int_0^Ss\Big(\Vert\nabla\dot{\Phi}\Vert_{L^2}^2+\Vert\nabla\Phi_s\Vert_{L^2}^2\Big)ds\le I(S).\label{(2.75**)}
\end{align}
\end{Lemma}

\begin{proof}
Due to $(\ref{MHD})_3$, we have
\begin{align}
\dot{\Phi}-\frac{\kappa}{\mathbb{C}_v}\varrho^{-1}\Delta\Phi+\frac{\mathcal{R}}{\mathbb{C}_v}(\nabla\cdot \mathbf{u})\Phi=\frac{2\iota\vert\mathfrak{D}(\mathbf{u})\vert^2+\aleph(\nabla\cdot \mathbf{u})^2+\nu\vert\rot \mathbb{H}\vert^2}{\mathbb{C}_v\varrho}\ge0,\label{theta(x,s)' estiamtes}
\end{align}
coupled with (\ref{(2.32**)}), we can prove (\ref{(2.74**)}) immediately.

Then, operating $\partial_s+\nabla\cdot(\mathbf{u}\cdot)$ to $(\ref{MHD})_3$, multiplying it by $\dot{\Phi}$ and integrating the result with (\ref{(2.45**)}), it yields
\begin{align}
&\frac{\mathbb{C}_v}{2}\Vert\dot{\Phi}\Vert_{L^2}^2+\kappa\Vert\nabla\dot{\Phi}\Vert_{L^2}^2\notag\\
&\quad=\int\kappa\Big((\nabla\cdot \mathbf{u})\Delta\Phi
-\partial_i(\partial_i\mathbf{u}\cdot\nabla\Phi)-\partial_i\mathbf{u}\cdot\nabla\partial_i\Phi\Big)\dot{\Phi}dx\notag\\
&\qquad+\int \mathcal{R}\varrho\Big(\Phi\partial_k\mathbf{u}^l\partial_l\mathbf{u}^k-\dot{\Phi}(\nabla\cdot \mathbf{u})
-\Phi\nabla\cdot \dot{\mathbf{u}}\Big)\dot{\Phi}dx\notag\\
&\qquad+\int\iota\Big(\partial_j\mathbf{u}^i+\partial_i\mathbf{u}^j\Big)\Big(\partial_i\dot{\mathbf{u}}^j
+\partial_j\dot{\mathbf{u}}^i
-\partial_i\mathbf{u}^k\partial_k\mathbf{u}^j-\partial_j\mathbf{u}^k\partial_k\mathbf{u}^i\Big)\dot{\Phi}dx\notag\\
&\qquad+\int2\aleph\Big(\nabla\cdot\dot{\mathbf{u}}-\partial_k\mathbf{u}^l\partial_l\mathbf{u}^k\Big)(\nabla\cdot \mathbf{u})\dot{\Phi}dx
+\int\Big(2\iota\vert\mathfrak{D}(\mathbf{u})\vert^2
+\aleph(\nabla\cdot \mathbf{u})^2\Big)(\nabla\cdot \mathbf{u})\dot{\Phi}dx\notag\\
&\qquad+\int2\nu(\rot \mathbb{H})\Big(\rot \dot{\mathbb{H}}-\nabla(\mathbf{u}\cdot\nabla)\times \mathbb{H}\Big)\dot{\Phi}dx
+\int\nu\Big(\vert\rot \mathbb{H}\vert^2(\nabla\cdot\mathbf{u})\Big)\dot{\Phi}dx\notag\\
&\quad\stackrel{\triangle}{=}\sum_{i=1}^{7}\mathcal{F}_i.\label{(3.95*)}
\end{align}
For $\mathcal{F}_{1}-\mathcal{F}_{5}$, we can see \cite[Lemma 2.8]{HJR}. So, we need analysis the estimates
of $\mathcal{F}_{6}$ and $\mathcal{F}_{7}$, using (\ref{GN11}), (\ref{(2.5**)}) and (\ref{(2.50**)}) obtains that
\begin{align}
(\mathcal{F}_{6}+\mathcal{F}_{7})
&\le I\Big(\Vert\nabla \mathbb{H}\Vert_{L^3}\Vert\nabla\dot{\mathbb{H}}\Vert_{L^2}\Vert\dot{\Phi}\Vert_{L^6}
+\Vert\nabla \mathbf{u}\Vert_{L^2}\Vert\nabla \mathbb{H}\Vert_{L^6}^2\Vert\dot{\Phi}\Vert_{L^6}\Big)\notag\\
&\le \frac{\kappa}{8}\Vert\nabla\dot{\Phi}\Vert_{L^2}^2
+I\Big(\Vert\nabla\dot{\mathbb{H}}\Vert_{L^2}^2+\Vert\nabla^2 \mathbb{H}\Vert_{L^2}^4\Big).\label{F 6,7' estiamtes}
\end{align}

Based on (\ref{(2.45**)}), substituting $\mathcal{F}_{1}-\mathcal{F}_{7}$ into (\ref{(3.95*)}), it gets
\begin{align}
\frac{d}{ds}\Vert\dot{\Phi}\Vert_{L^2}^2+\Vert\nabla\dot{\Phi}\Vert_{L^2}^2
&\le I\Vert\dot{\Phi}\Vert_{L^2}^2\Big(1+\Vert\nabla^2\Phi\Vert_{L^2}^2\Big)
+I\Big(1+\Vert\nabla^2\Phi\Vert_{L^2}^2+\Vert\nabla^2 \mathbf{u}\Vert_{L^2}^2\Big)\notag\\
&\quad+I\Big(\Vert\nabla\dot{\mathbf{u}}\Vert_{L^2}^2+\Vert\nabla\dot{\mathbb{H}}\Vert_{L^2}^2
+\Vert\nabla^2 \mathbf{u}\Vert_{L^2}^4+\Vert\nabla^2 \mathbb{H}\Vert_{L^2}^4\Big),\label{L 2.8**}
\end{align}
which is combined with (\ref{(2.46**)}), (\ref{(2.50**)}) and Gronwall' inequality, we have
\begin{align}
\sup_{0\le s\le S}\Big(s\Vert\dot{\Phi}\Vert_{L^2}^2\Big)+\int_0^Ss\Vert\nabla\dot{\Phi}\Vert_{L^2}^2ds
&\le I+I\sup_{0\le s\le S}\Big(s\Vert\nabla^2 \mathbf{u}\Vert_{L^2}^2\Big)\int_0^S\Vert\nabla^2 \mathbf{u}\Vert_{L^2}^2ds\notag\\
&\quad+I\sup_{0\le s\le S}\Big(s\Vert\nabla^2 \mathbb{H}\Vert_{L^2}^2\Big)\int_0^S\Vert\nabla^2 \mathbb{H}\Vert_{L^2}^2ds\notag\\
&\le I.\label{(2.77**)}
\end{align}

Moveover, utilizing (\ref{(2.46**)}), (\ref{(2.50**)}) and (\ref{(2.68**)}), we can prove
$s\Vert\nabla^2\Phi\Vert_{L^2}^2\in L^{\infty}(0,S)$.
From (\ref{theta_s' estimates}), we also can get
$\sqrt{s}\Phi\in L^{\infty}(0,S; L^2)\cap L^{\infty}(0,S; H^1).$
Subsequently, we can finish the proof of (\ref{(2.75**)}).

\end{proof}

\section{\large\bf The proof of Theorem \ref{Main result}}

In the eventual segment, we provide the proof of Theorem \ref{Main result}.
Similar to the way of \cite{{ATT},{HCG},{XJG}}, we can obtain the global existence.
The thought of proofing uniqueness problem is derived from \cite{{PW},{JJJO},{MR},{HJR}}. The detailed process is as follows:

\noindent\textbf{Proof of uniqueness.}
Assume that (\ref{MHD}) has two solutions $(\varrho_1,\mathbf{u}_1,\Phi_1,\mathbb{H}_1)$ and $(\varrho_2,\mathbf{u}_2,\Phi_2,\mathbb{H}_2)$,  which satisfies (\ref{(1.19**)}) and the same initial data on $\mathrm{R}^3\times[0,S]$. Define
\begin{align*}
\widehat{\varrho}\stackrel{\triangle}{=}\varrho_1-\varrho_2,\quad
\widehat{\mathbf{u}}\stackrel{\triangle}{=}\mathbf{u}_1-\mathbf{u}_2,\quad
\widehat{\Phi}\stackrel{\triangle}{=}\Phi_1-\Phi_2,\quad
\widehat{\mathbb{H}}\stackrel{\triangle}{=}\mathbb{H}_1-\mathbb{H}_2.
\end{align*}

Due to $(\ref{MHD})_1$, we reach that
\begin{align}
\widehat{\varrho}_s+\mathbf{u}_2\cdot\nabla \widehat{\varrho}+\widehat{\varrho}(\nabla\cdot\mathbf{u})_2+\varrho_1(\nabla\cdot\mathbf{u})
+\widehat{\mathbf{u}}\cdot\nabla\varrho_1=0,
\end{align}
after multiplying the result by $\widehat{\varrho}$ in $L^2$ and integrating by parts, one has
\begin{align}\label{(3.1**)}
\frac{d}{ds}\Vert\widehat{\varrho}\Vert_{L^2}^2
&\le I\Vert\nabla\cdot{\mathbf{u}_2}\Vert_{L^{\infty}}\Vert\widehat{\varrho}\Vert_{L^2}
+I\Vert\nabla\widehat{\mathbf{u}}\Vert_{L^2}\Vert\widehat{\varrho}\Vert_{L^2}.
\end{align}
Using (\ref{(1.19**)}), $\Vert\nabla\cdot{\mathbf{u}_2}\Vert_{L^{\infty}}\in L^1(0,S)$. Because of (\ref{(3.1**)}) and Gronwall's inequality,
\begin{align}\label{(3.2**)}
\Vert\widehat{\varrho}(s)\Vert_{L^2}\le Is^{\frac{1}{2}}\Big(\int_0^s\Vert\nabla\widehat{\mathbf{u}}\Vert_{L^2}^2dh\Big)^{\frac{1}{2}}, \quad \forall s\in[0,S].
\end{align}

Next, in terms of $(\ref{MHD})_2$ and $(\ref{MHD})_4$, we know ($\dot{\mathbf{u}}_2\stackrel{\triangle}{=}{\mathbf{u}}_{2s}+{\mathbf{u}}_2\cdot\nabla{\mathbf{u}}_2$)
\begin{align}\label{(3.3**)}
&\varrho_1\widehat{\mathbf{u}}_s+\varrho_1\mathbf{u}_1\cdot\nabla\widehat{\mathbf{u}}
-\iota\Delta\widehat{\mathbf{u}}-(\iota+\aleph)\nabla(\nabla\cdot\widehat{\mathbf{u}})\notag\\
&\quad=-\widehat{\varrho}\dot{\mathbf{u}}_2-\varrho_1\widehat{\mathbf{u}}\cdot\nabla\mathbf{u}_2-
\nabla\Big(\mathfrak{P}(\varrho_1,\Phi_1)-\mathfrak{P}(\varrho_2,\Phi_2)\Big)\\
&\qquad+\mathbb{H}_1\cdot\nabla\widehat{\mathbb{H}}+\widehat{\mathbb{H}}\cdot\nabla\mathbb{H}_2
-\widehat{\mathbb{H}}(\nabla\cdot\mathbf{u}_1)-\mathbb{H}\nabla\cdot\widehat{\mathbf{u}},\notag
\end{align}
and
\begin{align}\label{(3.3**)H}
\widehat{\mathbb{H}}_s-\nu\Delta\widehat{\mathbb{H}}
=\widehat{\mathbb{H}}\cdot\nabla\mathbf{u}_1
+\mathbb{H}_2\cdot\nabla\widehat{\mathbf{u}}-\mathbf{u}_1\cdot\nabla\widehat{\mathbb{H}}
-\widehat{\mathbf{u}}\cdot\nabla\mathbb{H}_2-\widehat{\mathbb{H}}(\nabla\cdot{\mathbf{u}}_1)
-\mathbb{H}_2(\nabla\cdot{\widehat{\mathbf{u}}}).
\end{align}
Multiplying (\ref{(3.3**)}) by ${\widehat{\mathbf{u}}}$, (\ref{(3.3**)H}) by $\widehat{\mathbb{H}}$ in $L^2$ and integrating by parts, after adding up the results, we get
\begin{align}\label{(3.3)H}
&\frac{1}{2}\frac{d}{ds}(\Vert\sqrt{\varrho_1}\widehat{\mathbf{u}}\Vert_{L^2}^2
+\Vert\widehat{\mathbb{H}}\Vert_{L^2}^2)+\iota\Vert\nabla\widehat{\mathbf{u}}\Vert_{L^2}^2
+(\iota+\aleph)\Vert\nabla\cdot\widehat{\mathbf{u}}\Vert_{L^2}^2
+\nu\Vert\nabla\widehat{\mathbb{H}}\Vert_{L^2}^2\notag\\
&\quad\le I\Vert\widehat{\varrho}\Vert_{L^2}\Vert\dot{\mathbf{u}}_2\Vert_{L^3}\Vert\widehat{\mathbf{u}}\Vert_{L^6}
+I\Vert\widehat{\mathbf{u}}\Vert_{L^2}\Vert\nabla\mathbf{u}_2\Vert_{L^3}
\Vert\widehat{\mathbf{u}}\Vert_{L^6}\notag\\
&\qquad+I(\Vert\Phi_2\Vert_{L^{\infty}}\Vert\widehat{\varrho}\Vert_{L^2}
+\Vert\widehat{\Phi}\Vert_{L^2})\Vert\nabla\widehat{\mathbf{u}}\Vert_{L^2}
+I\Big(\Vert\mathbb{H}_1\Vert_{L^3}+\Vert\mathbb{H}_1\Vert_{L^3}\Big)
\Vert\widehat{\mathbb{H}}\Vert_{L^6}\Vert\nabla\widehat{\mathbf{u}}\Vert_{L^2}\notag\\
&\qquad+I\Big(\Vert\widehat{\mathbb{H}}\Vert_{L^2}\Vert\nabla\mathbf{u}_1\Vert_{L^3}
+\Vert\mathbb{H}_2\Vert_{L^3}\Vert\nabla\widehat{\mathbf{u}}\Vert_{L^2}
+\Vert\mathbf{u}_1\Vert_{L^3}\Vert\nabla\widehat{\mathbb{H}}\Vert_{L^2}
+\Vert\widehat{\mathbf{u}}\Vert_{L^3}\Vert\nabla\widehat{\mathbb{H}}\Vert_{L^2}
\Big)\Vert\widehat{\mathbb{H}}\Vert_{L^6}\notag\\
&\quad\le\frac{\iota}{2}\Vert\nabla\widehat{\mathbf{u}}\Vert_{L^2}^2
+\frac{\nu}{2}\Vert\nabla\widehat{\mathbb{H}}\Vert_{L^2}^2
+I\Big(1+\Vert\dot{\mathbf{u}}_2\Vert_{L^3}^2+\Vert\Phi_2-1\Vert_{L^{\infty}}\Big)
\Vert\widehat{\Phi}\Vert_{L^2}^2\notag\\
&\qquad+I\Big(\Vert\widehat{\Phi}\Vert_{L^2}^2+\Vert\widehat{\mathbf{u}}\Vert_{L^2}^2
+\Vert\widehat{\mathbb{H}}\Vert_{L^2}^2\Big)
\end{align}
namely, it reaches that
\begin{align}\label{(3.4**)}
&\frac{d}{ds}(\Vert\sqrt{\varrho_1}\widehat{\mathbf{u}}\Vert_{L^2}^2
+\Vert\widehat{\mathbb{H}}\Vert_{L^2}^2)
+\Vert\nabla\widehat{\mathbf{u}}\Vert_{L^2}^2+\Vert\nabla\widehat{\mathbb{H}}\Vert_{L^2}^2\notag\\
&\quad\le I\Big(1+\Vert\dot{\mathbf{u}}_2\Vert_{L^3}^2+\Vert\Phi_2-1\Vert_{L^{\infty}}\Big)
\Vert\widehat{\Phi}\Vert_{L^2}^2+I\Big(\Vert\widehat{\Phi}\Vert_{L^2}^2+\Vert\widehat{\mathbf{u}}\Vert_{L^2}^2
+\Vert\widehat{\mathbb{H}}\Vert_{L^2}^2\Big).
\end{align}

Then, one may check that from $(\ref{MHD})_3$ ($\dot{\Phi}_2\stackrel{\triangle}{=}{\Phi}_{2s}+\mathbf{u}\cdot\nabla\Phi_2$)
\begin{align*}
&\mathbb{C}_v(\varrho_1\widehat{\Phi}_s+\varrho_1\mathbf{u}_1\cdot\nabla\Phi)-\kappa\Delta\Phi\\
&\quad=-\mathbb{C}_v(\widehat{\varrho}\dot{\Phi}_2+\varrho_1\widehat{\mathbf{u}}\cdot\nabla\Phi_2)
-\Big(\mathfrak{P}(\varrho_1,\Phi_1)(\nabla\cdot\mathbf{u}_1)-\mathfrak{P}(\varrho_2,\Phi_2)
(\nabla\cdot\mathbf{u}_2)\Big)\\
&\qquad+2\iota\Big(\vert\mathfrak{D}(\mathbf{u}_1)\vert^2-\vert\mathfrak{D}(\mathbf{u}_2)\vert^2\Big)
+\aleph\Big(\vert(\nabla\cdot\mathbf{u}_1)\vert^2-\vert(\nabla\cdot\mathbf{u}_2)\vert^2\Big)
+\nu\Big(\vert\rot \mathbb{H}_1\vert^2-\vert\rot \mathbb{H}_2\vert^2\Big),
\end{align*}
compared with similar process like (3.5) in \cite{HJR} and (\ref{(3.3)H}), we have
\begin{align}
\int(\vert\nabla\mathbb{H}_1\vert+\vert\nabla\mathbb{H}_2\vert)
\vert\nabla\widehat{\mathbb{H}}\vert\vert\widehat{\Phi}\vert dx
&\le I(\Vert\nabla\mathbb{H}_1\Vert_{L^3}+\Vert\nabla\mathbb{H}_2\Vert_{L^3})
\Vert\nabla\widehat{\mathbb{H}}\Vert_{L^2}\Vert\widehat{\Phi}\Vert_{L^6}\notag\\
&\le I(\Vert\nabla\widehat{\mathbb{H}}\Vert_{L^2}^2+\Vert\nabla\widehat{\Phi}\Vert_{L^2}^2) \notag
\end{align}
and
\begin{align}\label{(3.5**)}
&\frac{d}{ds}\Vert\sqrt{\varrho_1}\widehat{\Phi}\Vert_{L^2}^2+\Vert\nabla\widehat{\Phi}\Vert_{L^2}^2\notag\\
&\quad\le I(\Vert\nabla\widehat{\mathbf{u}}\Vert_{L^2}^2+\Vert\nabla\widehat{\mathbb{H}}\Vert_{L^2}^2)
+I\Big(1+\Vert\dot{\Phi}_2\Vert_{L^3}^2+\Vert\Phi_2-1\Vert_{L^{\infty}}\Big)
\Vert\widehat{\Phi}\Vert_{L^2}^2\\
&\qquad+I(1+\Vert\nabla{\Phi}_2\Vert_{L^3}^2)(\Vert\widehat{\Phi}\Vert_{L^2}^2
+\Vert\widehat{\mathbf{u}}\Vert_{L^2}^2). \notag
\end{align}

Due to (\ref{(3.4**)}), (\ref{(3.5**)}) and (\ref{(3.2**)}), we get
\begin{align}\label{(3.6**)}
&\frac{d}{ds}(\Vert\sqrt{\varrho_1}\widehat{\mathbf{u}}\Vert_{L^2}^2
+\Vert\widehat{\mathbb{H}}\Vert_{L^2}^2+\Vert\sqrt{\varrho_1}\widehat{\Phi}\Vert_{L^2}^2)
+\Vert\nabla\widehat{\mathbf{u}}\Vert_{L^2}^2+\Vert\nabla\widehat{\mathbb{H}}\Vert_{L^2}^2
+\Vert\nabla\widehat{\Phi}\Vert_{L^2}^2\notag\\
&\quad\le Is(1+\Vert\dot{\mathbf{u}}_2\Vert_{L^3}^2+\Vert\dot{\Phi}_2\Vert_{L^3}^2+\Vert\Phi_2-1\Vert_{L^{\infty}})
\int_0^s\Vert\nabla\widehat{\mathbf{u}}\Vert_{L^2}^2dh\\
&\qquad+I(1+\Vert\nabla{\Phi}_2\Vert_{L^3}^2)(\Vert\widehat{\Phi}\Vert_{L^2}^2
+\Vert\widehat{\mathbf{u}}\Vert_{L^2}^2). \notag
\end{align}

Denoting
\begin{align*}
\mathrm{X}(s)\stackrel{\triangle}{=}(\Vert\sqrt{\varrho_1}\widehat{\mathbf{u}}\Vert_{L^2}^2
+\Vert\widehat{\mathbb{H}}\Vert_{L^2}^2+\Vert\sqrt{\varrho_1}\widehat{\Phi}\Vert_{L^2}^2)
+\int_0^s(\Vert\nabla\widehat{\mathbf{u}}\Vert_{L^2}^2+\Vert\nabla\widehat{\mathbb{H}}\Vert_{L^2}^2
+\Vert\nabla\widehat{\Phi}\Vert_{L^2}^2)dh.
\end{align*}
By (\ref{(2.45**)}) and (\ref{(3.6**)}), one reaches that
\begin{align}\label{(3.7**)}
\mathrm{X}'(s)\le\mathrm{Y}(s)\mathrm{X}(s) \quad with\quad \mathrm{X}(0)=0,
\end{align}
where
\begin{align*}
\mathrm{Y}(s)&\stackrel{\triangle}{=}
Is(1+\Vert\dot{\mathbf{u}}_2\Vert_{L^3}^2+\Vert\dot{\Phi}_2\Vert_{L^3}^2+\Vert\Phi_2-1\Vert_{L^{\infty}})
+I(1+\Vert\nabla{\Phi}_2\Vert_{L^3}^2)\\
&\le Is(1+\Vert\dot{\mathbf{u}}_2\Vert_{L^2}^2+\Vert\nabla\dot{\mathbf{u}}_2\Vert_{L^2}^2
+\Vert\dot{\Phi}_2\Vert_{L^2}^2+\Vert\nabla\dot{\Phi}_2\Vert_{L^2}^2)+I(1+\Vert\nabla{\Phi}_2\Vert_{H^1}^2).
\end{align*}
Which is connected with (\ref{(1.19**)}), (\ref{(3.7**)}) and Gronwall's inequality $\mathrm{X}(s)=0$, i.e.,
\begin{align*}
\mathrm{X}(s)=0, \quad \forall s\in[0,S].
\end{align*}

Subsequently, we know $(\widehat{\mathbf{u}},\widehat{\mathbb{H}},\widehat{\Phi})(x,s)=0$ a.e. on $\mathrm{R}^3\times[0,S]$. And using (\ref{(3.2**)}) again, we can obtain $\widehat{\varrho}(x,s)=0$ a.e. on $\mathrm{R}^3\times[0,S]$. The whole proof is over.

~\\

\end{document}